\documentclass[preprint,12pt]{elsarticle}

\usepackage{amssymb, bm}
\usepackage{amsthm, amsmath}

\usepackage{graphicx,color}
\usepackage{caption}
\usepackage{subcaption}
\usepackage{multirow}
\usepackage{mathtools,xparse}
\usepackage{algorithm}
\usepackage{algorithmic}

\usepackage{listings}

\newtheorem{lemma}{Lemma}
\newtheorem{thm}{Theorem}
\newtheorem{prop}{Proposition}

\makeatletter
\newenvironment{smallermatrix}[1][c]
{\null\,\vcenter\bgroup
  \Let@\restore@math@cr\default@tag
  \baselineskip0pt \lineskip0.4pt \lineskiplimit0pt
  \ialign\bgroup\if#1l\else\hfil\fi$\m@th\scriptstyle##$\if#1r\else\hfil\fi&&\thickspace\hfil
  $\m@th\scriptstyle##$\hfil\crcr
}{%
  \crcr\egroup\egroup\,%
}
\makeatother

\NewDocumentCommand{\ts}{O{c} e{^?_}}{
  \begin{smallermatrix}[#1]
  \mathstrut\IfValueT{#2}{#2} \\
  \mathstrut\IfValueT{#3}{#3} \\
  \mathstrut\IfValueT{#4}{#4}
  \end{smallermatrix}%
}

\journal{Arxiv}

\begin{document}

\begin{frontmatter}



\title{High-order empirical interpolation methods for real-time solution of parametrized nonlinear PDEs}


% \author[inst1]{Ngoc Cuong Nguyen}

% \affiliation[inst1]{organization={Department of Aeronautics and Astronautics, Massachusetts Institute of Technology},%Department and Organization
%             addressline={77 Massachusetts Avenue}, 
%             city={Cambridge},
%             postcode={02139}, 
%             state={Massachusetts},
%             country={United States}}

\author[inst2]{Ngoc Cuong Nguyen}
%\author[inst2]{Jaime Peraire}

\affiliation[inst2]{organization={Center for Computational Engineering, Department of Aeronautics and Astronautics, Massachusetts Institute of Technology},%Department and Organization
            addressline={77 Massachusetts
Avenue}, 
            city={Cambridge},
            state={MA},
            postcode={02139}, 
            country={USA}}
            
% \affiliation[inst3]{organization={Center for Computing Research, Sandia National Laboratories},%Department and Organization
%             city={Albuquerque},
%             state={NM},
%             postcode={87185}, 
%             country={USA}}
%\affiliation[inst3]{organization={Computational Multiscale Department, Sandia National Laboratories},%Department and Organization
%            addressline={P.O. Box 5800, MS
%1322}, 
%            city={Albuquerque},
%            postcode={87185}, 
%            state={NM},
%            country={USA}}

\begin{abstract}
We present novel model reduction methods for rapid solution of parametrized nonlinear partial differential equations (PDEs) in real-time or many-query contexts. Our approach combines reduced basis (RB) space for rapidly convergent approximation of the parametric solution manifold, Galerkin projection of the underlying PDEs onto the RB space for dimensionality reduction, and high-order empirical interpolation for efficient treatment of the nonlinear terms. We propose a class of high-order empirical interpolation methods to derive basis functions and interpolation points by using high-order partial derivatives of the nonlinear terms. As these methods can generate  high-quality basis functions and interpolation points from a snapshot set of full-order model (FOM) solutions, they  significantly improve the approximation accuracy. We develop effective {\em a posteriori} estimator to quantify the interpolation errors and construct a parameter sample via greedy sampling. Furthermore, we implement two hyperreduction schemes to construct efficient reduced-order models: one that applies the empirical interpolation before Newton's method and another after. The latter scheme shows flexibility in controlling hyperreduction errors. Numerical results are presented to demonstrate the accuracy and efficiency of the proposed methods.
\end{abstract}


%\item Efficient reduced-order modeling for nonlinear PDEs in real-time applications
%
%\item High-order interpolation improves accuracy using higher-order partial derivatives.
%
%\item Two hyperreduction schemes are developed to construct efficient ROMs.
%
%\item Greedy sampling enhances ROM accuracy and robustness.


\end{frontmatter}


\section{Introduction}
\label{sec:intro}

Numerous applications in engineering and science require repeated solutions of parametrized partial differential equations (PDEs) particularly in the context of design, optimization, control, uncertainty quantification. Numerical approximation of the parametrized PDEs can be achieved by classical discretization methods such as finite element (FE), finite difference (FD) and finite volume (FV) methods, which will be referred to as {\em full order models} (FOMs). Since FOMs often require a large number of degrees of freedom to achieve the desired  accuracy, computing many FOM solutions can be too costly. Alleviating this computational burden is the main motivation behind the development of {\em reduced order models} (ROMs) for parametrized PDEs. 

The construction and deployment of ROMs involve two distinct stages \cite{ARCME}. During the offline stage, the parametrized PDE is solved at selected parameter values to obtain an ensemble of Full Order Model (FOM) solutions, which is used to construct a Reduced Basis (RB) space for capturing the  solution manifold of the parametrized PDE. A Galerkin or Petrov-Galerkin projection is then performed to project the FOM onto the RB space to create a ROM. If the PDE is linear in the field variable and affine in the parameter, parameter-independent matrices and vectors of the ROM system can be precomputed and stored. In the online stage, the ROM system can be assembled and solved efficiently for any new parameter values, benefiting from the heavy computational workload completed during the offline stage. This RB approach significantly reduces the computational cost of  the online stage and retains the high accuracy of the FOM. RB methods have been widely used to achieve rapid and accurate solutions of parametrized PDEs \cite{Chen2009, Chen2010a, deparis07, EftangPateraRonquist10,Eftang2012a, Huynh2007a, Huynh2010, PhuongHuynh2013, Karcher2018, Knezevic2011, Knezevic2010,  nguyen04:_handb_mater_model, Nguyen2007, Nguyen_SantaFE08, Calcolo, ARCME, Rozza05_apnum, Sen2006b, veroy04:_certif_navier_stokes, Veroy2002, Vidal-Codina2014, Vidal-Codina2018a}. 

For nonaffine and nonlinear PDEs, an efficient offline-online decomposition can be difficult in the presence of nonaffine and nonlinear terms, often leading to computationally expensive ROMs during the online stage. Efficient treatment of nonffine and nonlinear terms is required to construct an efficient ROM. This can be achieved through various hyperreduction techniques like  empirical interpolation methods (EIM) \cite{Barrault2004,Grepl2007,Nguyen2007,Drohmann2012,Manzoni2012,Hesthaven2014,Hesthaven2022,Chen2021,Eftang2010b}, best-points interpolation method \cite{Nguyen2008a,Nguyen2008d}, generalized empirical interpolation method \cite{Maday2015a,Maday2013},  empirical quadrature methods (EQM) \cite{Patera2017,Hernandez2017,Yano2019a}, energy-conserving sampling and weighting \cite{Farhat2014, Farhat2015}, gappy-POD \cite{everson95karhunenloeve,Galbally2010}, and integral interpolation methods \cite{Carlberg2011,Chaturantabut2011,Drohmann2012, Kerfriden2013,Radermacher2016},  While hyperreduction techniques provide an efficient approximation of the nonaffine and nonlinear terms, they involve additional offline computations, such as calculating extra FOM solutions and constructing interpolation points or quadrature points. These computations, although adding to the offline stage's complexity, ensure that the online stage remains efficient by rendering the online computational complexity independent of the size of the FOM. 

This paper presents model reduction methods for solving parametrized nonlinear PDEs in the real-time or many-query contexts mentioned above. The approach builds on our recent work on the first-order empirical interpolation method \cite{Nguyen2023d,Nguyen2024} and extends it to a class of high-order empirical interpolation methods.  Using high-order partial derivatives of the nonlinear terms, these methods can generate high-quality basis functions and interpolation points from a snapshot set of FOM solutions to significantly improve the approximation accuracy compared to the original EIM method \cite{Barrault2004,Grepl2007,Maday2008c}. We develop an effective {\em a posteriori} estimator to quantify the interpolation errors and construct a parameter sample via greedy sampling. Furthermore, we introduce two distinct hyperreduction schemes to construct efficient ROMs. The first scheme applies hyperreduction to the nonlinear system, whereas the second scheme applies hyperreduction to the linearized system. The second scheme has the potential to eliminate hyperreduction errors and thus provides high accuracy for hyperreduced ROMs. Numerical examples are presented to compare the performance  of various interpolation methods and several model reduction methods. 
 
The paper is organized as follows. We describe model reduction methods for parammetrized nonlinear PDEs in Section 2. In Section 3, we introduce a class of high-order empirical interpolation methods for efficient RB treatment of nonlinear terms. Numerical results are presented in Section 4 to demonstrate our approach. Finally, in Section 5, we provide concluding remarks and suggest directions for future work.

\section{Model Reduction Methods for Parametrized Nonlinear PDEs}

In this section, we describe model reduction methods for parametrized nonlinear PDEs. The traditional Galerkin-Newton method employs Galerkin projection combined with Newton’s method to solve the reduced-order nonlinear system. However, despite its reduced dimensionality, the evaluation of nonlinear terms remains costly due to their dependence on the FOM dimension. We employ two hyperreduction schemes to reduce this computational expense. The first scheme applies hyperreduction before Newton’s method to permit an efficient offline-online computational decomposition. The second scheme applies hyperreduction after Newton’s method to reduce the computational cost of the Newton iterations. Both schemes use empirical interpolation methods to approximate nonlinear terms. The second scheme allows for more flexible and efficient treatment of nonlinear terms.

\subsection{Finite element approximation}

Many applications require repeated, reliable, and real-time predictions of selected performance metrics, or "outputs" $s^{\rm e}$. Here, the superscript "e" denotes "exact," and we will later introduce a full order model (FOM) without a superscript. These outputs are typically functionals of a field variable, $u^{\rm e}(\bm \mu)$, associated with a parametrized partial differential equation (PDE) that describes the underlying physics. The input parameters $\bm \mu$ identify a specific configuration of the component or system, including geometry, material properties, boundary conditions, and loads. 

The abstract formulation can be stated
as follows: given any $\bm \mu \in {\cal D} \subset \mathbb{R}^P$, we
evaluate $s^{\rm e}(\bm \mu) = \ell^O(u^{\rm e}(\bm \mu);\bm \mu)$, where $u^{\rm e}(\bm \mu) \in
X^{\rm e}$ is the solution of
\begin{equation}
a(u^{\rm e}(\bm \mu),v;\bm \mu) = \ell(v; \bm \mu), \quad \forall v \in X^{\rm e}.  
\label{eq:1}
\end{equation}
Here $\mathcal{D}$ is the parameter domain in which our $P$-tuple parameter $\bm \mu$ resides; $X^{\rm e}(\Omega)$ is an appropriate Hilbert space;
$\Omega$ is a bounded domain in $\mathbb{R}^D$ with Lipschitz continuous boundary $\partial \Omega$; $\ell(\cdot;\bm \mu), \ \ell^O(\cdot;\bm \mu)$ are $X^{\rm
  e}$-continuous linear functionals;  and $a(\cdot,\cdot;\bm \mu)$ is a variational form of the parameterized PDE operator. We assume that both $\ell$ and $\ell^O$ are independent of $\bm \mu$. For many second-order PDEs, the variational form $a$ may be expressed as
\begin{equation}
a(w,v; \bm \mu) = \sum_{q=1}^{Q} \Theta^q(\bm \mu) a^q(w,v) + b(w,v; \bm \mu)  
\end{equation} 
where $a^q(\cdot, \cdot)$ are $\bm \mu$-independent bilinear forms, and $\Theta^q(\bm \mu)$ are $\bm \mu$-dependent functions. The variational form $b$ is a nonaffine and nonlinear form which is assumed to be 
\begin{equation}
b(w,v; \bm \mu) = \int_{\Omega} g(w, \bm \mu) v \, d  \bm x + \int_{\Omega} \bm f(w, \bm \mu) \cdot \nabla v \, d  \bm x
\end{equation} 
where $g$ and $\bm f$ are a scalar and vector-valued nonlinear function of $(w,\bm \mu)$, respectively. 

%The problem is said to be {\em affine} in the parameter if these functions are linear .

In actual practice, we replace $u^{\rm e}(\bm \mu)$ with a FOM solution, $u(\bm \mu)$,
which resides in a finite element
approximation space $X \subset X^{\rm e}$ of {\em very} large dimension $\mathcal{N}$ and satisfies
\begin{equation}
a(u(\bm \mu),v;\bm \mu) = \ell(v), \quad \forall v \in X.  
\label{eq:1a}
\end{equation}
We then evaluate the FOM output as
\begin{equation}
s(\bm \mu) = \ell^O(u(\bm \mu)) \ .
\label{eq:1p}
\end{equation}
We shall assume the FOM discretization is sufficiently rich such that $u(\bm \mu)$ and $u^{\rm e}(\bm \mu)$
and hence $s(\bm \mu)$ and $s^{\rm e}(\bm \mu)$ are indistinguishable at the accuracy level of interest. 

%The RB approximation shall be built upon this FOM, and the RB error will thus be evaluated with respect to $u(\bm \mu) \in X$.

% where $a_l$ is an affine bilinear form and $a_n$ is a nonaffine bilinear or a nonlinear form. The affine bilinear form has the following affine 
% \begin{equation}
% a_l(w,v; \mu) = \sum_{q=1}^{Q} \Theta_q(\mu) a_q(w,v),
% \end{equation} 
% where the bilinear forms $a_q$ are independent of $\mu$ and $\Theta_q(\mu)$ 

% In this section, we assume that $a(\cdot,\cdot;\mu)$ is a
% $X^{\rm e}$-continuous bilinear form. In the next section, we extend our approach to nonlinear PDEs in which $a$ is a nonlinear operator in the first argument. The problem is said to be {\em affine} in the parameter if both $a$ and $f$ can be expressed as
% \begin{equation}
% a(w,v; \mu) = \sum_{q=1}^{Q^a} \Theta_q^a(\mu) a_q(w,v), \quad f(v; \mu) = \sum_{q=1}^{Q^f} \Theta_q^f(\mu) f_q(v)
% \end{equation} 
% where $a_q$ and $f_q$ are independent of $\mu$. 

\subsection{Galerkin-Newton method}

%We begin with motivating the generative RB approach
We assume that we are given a parameter sample, $S_N = \{\bm \mu_1 \in {\cal D}
,\cdots, \bm \mu_N \in {\cal D}\}$, and associated RB space $W_N = \mbox{span} \{\zeta_j \equiv 
u(\bm \mu_j), \ 1 \leq j \leq N\}$, where
$u(\bm \mu_j)$ is the solution of~(\ref{eq:1a}) for $\bm \mu = \bm \mu_j$. We then orthonormalize the $\zeta_j, 1 \leq j \leq N,$ with respect to $(\cdot,\cdot)_X$ so that $(\zeta_i,\zeta_j)_X = \delta_{ij}, 1 \leq i,j \leq N$. The RB
approximation would be obtained by a standard Galerkin projection: given
$\bm \mu \in {\cal D}$, we evaluate $s_{N}(\bm \mu) = \ell^O(u_{N}(\bm \mu))$, where $u_{N}(\bm \mu) \in W_N$ is the solution of
\begin{equation}
a(u_{N}(\bm \mu),v; \bm \mu) =  \ell(v), \quad
\forall v \in W_N.   
\label{eq:6-24a}
\end{equation} 
We express $u_N(\bm \mu) = \sum_{j=1}^N \alpha_{N \, j}(\bm \mu) \zeta_j$ and choose
test functions $v = \zeta_n, \ 1 \leq n \leq N$, in~(\ref{eq:6-24a}), to find the coefficient vector $\bm \alpha_N(\bm \mu) \in \mathbb{R}^N$ as the solution of the following  system
\begin{equation}
\left(\sum_{q=1}^Q \Theta^q(\bm \mu) \bm A_N^q \right) \bm \alpha_N(\bm \mu) + \bm b_N(\bm \alpha_N(\bm \mu))  = \bm l_N .
\label{eq:6-24b}
\end{equation} 
Here $\bm A_N^q \in \mathbb{R}^{N \times N}, 1 \le q \le Q,$ and $\bm l_N \in \mathbb{R}^{N}$ have entries
\begin{equation}
A_{N, ij}^q  = a^q(\zeta_i,\zeta_j), \qquad  l_{N, i} =  \ell(\zeta_i) , \qquad 1 \leq i,j \leq N  ,
\label{eq:6-24f}
\end{equation} 
while $\bm b_N(\bm \alpha_N(\bm \mu))  \in \mathbb{R}^{N}$ have entries
\begin{equation}
\label{eq9b}
b_{N, i}(\bm \alpha_N(\bm \mu)) = \int_{\Omega} g(u_N(\bm \mu), \bm \mu) \zeta_i \, d \bm x + \int_{\Omega} \bm f(u_N(\bm \mu),  \bm \mu) \cdot \nabla \zeta_i \, d  \bm x .
\end{equation} 
Note that $\bm A_N^q, 1 \le q \le Q,$ and $\bm L_N$ can be pre-computed and stored in the offline stage, whereas $\bm b_N(\bm \alpha_N(\bm \mu))$ can not be pre-computed because of the nonaffine and nonlinear terms $g(u_N(\bm \mu), \bm \mu)$ and $\bm f(u_N(\bm \mu),  \bm \mu)$. 

%\subsection{Affine Linear PDEs}

If there are no nonaffine and nonlinear terms in the partial differential operators, meaning that $b(w, v) = 0, \forall w, v \in X$, then the problem is said to be an affine linear PDE. In this case, the system (\ref{eq:6-24b}) becomes 
\begin{equation}
\left(\sum_{q=1}^Q \Theta^q(\bm \mu) \bm A_N^q \right) \bm \alpha_N(\bm \mu)  = \bm l_N .
\label{eq:6-24c}
\end{equation} 
This system can be solved for $\bm \alpha_N(\bm \mu)$ within $O(QN^2 + N^3)$ operations for any $\bm \mu \in \mathcal{D}$. The RB output is then evaluated as $s_N(\bm \mu) = (\bm l_N^O)^T \bm \alpha_N(\bm \mu)$, where $l_{N,i}^O = \ell^O(\zeta_i), 1 \le i \le N,$ are pre-computed in the offline stage.

For non-affine and nonlinear PDEs, the system (\ref{eq:6-24b}) is a nonlinear system of equations.  To solve it, we use Newton's method to linearize it at a current iterate $\bar{\bm \alpha}_N(\bm \mu)$. Thus, we find the increment $\delta \bm \alpha_{N}(\bm\mu) \in \mathbb{R}^N$ as the solution of the following linear system
\begin{equation}
\left(\sum_{q=1}^Q \Theta^q(\bm \mu) \bm A_N^q + \bm B_N(\bar{\bm \alpha}_N(\bm \mu))   \right) \delta \bm \alpha_N(\bm \mu) = \bm r_N(\bar{\bm \alpha}_N(\bm \mu)) ,
\label{eq:6a}
\end{equation} 
where
\begin{equation}
\bm r_N(\bar{\bm \alpha}_N(\bm \mu)) = \bm l_N - \bm b_N(\bar{\bm \alpha}_N(\bm \mu)) - \left(\sum_{q=1}^Q \Theta^q(\bm \mu) \bm A_N^q \right) \bar{\bm \alpha}_N(\bm \mu) ,
\label{eq:6b}
\end{equation} 
and $\bm B_N(\bar{\bm \alpha}_N(\bm \mu)) \in \mathbb{R}^{N \times N}$ has entries
\begin{equation}
B_{N, ij}(\bar{\bm \alpha}_N(\bm \mu)) = \int_{\Omega} \left( g_u(\bar{u}_N(\bm \mu), \bm \mu) \zeta_i + \bm f_u(\bar{u}_N(\bm \mu),  \bm \mu) \cdot \nabla \zeta_i  \right) \zeta_j \, d  \bm x .
\end{equation} 
Here $g_u$ and $\bm f_u$ denote the partial derivatives of $g$ and $\bm f$ with respect to the first argument, respectively. Both the matrix $\bm B_N(\bar{\bm \alpha}_N(\bm \mu))$ and the vector $\bm b_N(\bar{\bm \alpha}_N(\bm \mu))$ must be computed at each Newton iteration since they depend on the current iterate $\bar{u}_N(\bm \mu)$. However, they are computationally expensive to form because each of their entries is an integral over the entire physical domain. In particular, there are $(N + N^2)$ integrals to be evaluated at each Newton iteration. Computing each of these integrals incurs a high computational cost that scales with the dimension of the FE approximation. Although the linear system (\ref{eq:6a}) is small, it is computationally intensive to solve. As a result, the RB method does not offer a significant speedup over the FE method.

\subsection{Hyperreduction--Galerkin-Newton method}

%We devise two different hyper-reduction strategies to recover efficient ROMs for nonaffine and nonlinear PDEs. 

The method applies hyperreduction to the nonlinear integrals in  the nonlinear system (\ref{eq:6-24b}) and uses Newton method to linearize the resulting system. To this end, we approximate the integrand $g(u_N(\bm \mu), \bm \mu)$ with the following function
\begin{equation}
g_M(\bm x, \bm \mu) = \sum_{m=1}^{M} \beta_m(\bm \mu) \psi_m (\bm x), 
\end{equation} 
where
\begin{equation}
\label{eq15b}
\sum_{m=1}^{M}  \psi_m (\bm y_k) \beta_m(\bm \mu) = g(u_N(\bm y_k, \bm \mu), \bm \mu), \qquad 1 \le k \le M .
%\bm \beta_M(\bm \mu) = (\bm B^{ g}_M)^{-1} \bm b_M^{ g}(\bm \mu) . 
\end{equation}
Here $\{\psi_m(\bm x)\}_{m=1}^M$ is a set of interpolation functions and $\{\bm y_m \in \Omega\}_{m=1}^M$ is a set of interpolation points. Similarly,  we approximate the integrand $\bm f(u_N(\bm \mu),  \bm \mu)$ with the following vector-valued function
\begin{equation}
f^d_K(\bm x, \bm \mu) = \sum_{k=1}^{K^d} \gamma^{d}_k(\bm \mu) \phi^{d}_k (\bm x),  \qquad 1 \le d \le D,  
\end{equation} 
where
\begin{equation}
\label{eq17b}
\sum_{k=1}^{K^d}  \phi^d_k (\bm z_m^d) \gamma^d_k(\bm \mu) = f^d(u_N(\bm z_m^d, \bm \mu), \bm \mu), \qquad 1 \le m \le K^d .
\end{equation}
Here $\{\phi_k^d(\bm x)\}_{k=1}^{K^d}$ is a set of interpolation functions and $\{\bm z^d_m \in \Omega\}_{k=1}^{K^d}$ is a set of interpolation points, which are used to define the interpolant $f^d_K(\bm x, \bm \mu)$ for the $d$th component of $\bm f(u_N(\bm \mu),  \bm \mu)$. We will discuss the construction of   interpolation functions and the selection of interpolation points in the next section.

Therefore, we can approximate the nonlinear term $\bm b_{N}(\bm \alpha_N(\bm \mu))$ in (\ref{eq9b}) with
\begin{equation}
\label{eq18b}
\tilde{b}_{N,i}(\bm \alpha_N(\bm \mu)) = \sum_{m=1}^{M} \beta_m(\bm \mu)  \int_{\Omega} \psi_m\zeta_i \, d \bm x + \sum_{d=1}^D \sum_{k=1}^{K^d} \gamma^{d}_k(\bm \mu)  \int_{\Omega} \phi_k^{d}  \frac{\partial  \zeta_i}{\partial x_d} \, d  \bm x .
\end{equation} 
Substituting (\ref{eq15b}) and (\ref{eq17b}) into (\ref{eq18b}) yields the following result in the matrix notation 
\begin{equation}
\label{eq19b}
\tilde{\bm b}_{N}(\bm \alpha_N(\bm \mu)) = \bm C_{NM} \, g( \bm Q_{M N} \bm \alpha_N(\bm \mu)) +  \sum_{d=1}^D  \bm D_{NK}^d f^d( \bm S^d_{K N} \bm \alpha_N(\bm \mu)) , 
\end{equation} 
where $\bm C_{NM} = \bm E_{NM} (\bm H_M)^{-1}$, $\bm D^d_{NK} = \bm F^d_{NK} (\bm G^d_K)^{-1}$, and
\begin{equation}
\begin{split}
E_{NM, im} & = \int_{\Omega} \psi_m\zeta_i \, d \bm x, \quad H_{M, m k} = \psi_m(\bm y_k), \quad  Q_{M N, m j} = \zeta_j(\bm y_m) , \\
F_{NK, ik}^d & = \int_{\Omega} \phi_k^{d}  \frac{\partial  \zeta_i}{\partial x_d} \, d  \bm x, \quad G^d_{K, m k} = \phi_k(\bm z_m^d), \quad  S^d_{K N, k j} = \zeta_j(\bm z^d_k) .
\end{split}
\end{equation} 
By replacing $\bm b_{N}(\bm \alpha_N(\bm \mu))$ with $\tilde{\bm b}_{N}(\bm \alpha_N(\bm \mu))$ in (\ref{eq:6-24b}), we obtain the following nonlinear system of equations 
\begin{multline}
\left(\sum_{q=1}^Q \Theta^q(\bm \mu) \bm A_N^q \right) \bm \alpha_N(\bm \mu) + \bm C_{NM} \, g( \bm Q_{M N} \bm \alpha_N(\bm \mu)) \\ +  \sum_{d=1}^D  \bm D_{NK}^d f^d( \bm S^d_{K N} \bm \alpha_N(\bm \mu))  = \bm l_N .
\label{eq21b}
\end{multline} 
Since this nonlinear system is purely algebraic and small, it can be solved efficiently by using Newton method. 

The offline and online stages of the Hyperreduction--Galerkin-Newton (H-GN) method are summarized in Algorithm 1 and Algorithm 2, respectively. The offline stage is expensive and performed once. All the quantities computed in the offline stage are independent of  $\bm \mu$.  In the online stage, the RB output $s_N(\bm \mu)$ is calculated for any $\bm \mu \in \mathcal{D}$. The computational cost of the online stage is $O(N^3 + (Q + M + \sum_{d=1}^D K^d) N^2)$ for each Newton iteration. 
Hence, as required in the many-query or real-time contexts, the online complexity is independent of $\mathcal{N}$, which is the dimension of the FOM. Thus, we expect computational savings of several orders of magnitude relative to both the FOM and the Galerkin-Newton method described earlier. 

%It is important to note that the computational complexity scales linearly with $M$.  Therefore, it is advantageous to increase $M$ to make the resulting ROM more accurate. %The online stage can be executed many times.  The online computational complexity of solving the system is $O(N^3 + (Q + M + \sum_{d=1}^D K^d) N^2)$ for each Newton iteration. 

\begin{algorithm}
\begin{algorithmic}[1]
\REQUIRE{The parameter sample set $S_N = \{\bm \mu_j, 1 \le j \le N\}$.}
\ENSURE{$\bm l_N, \bm l_N^O, \bm A_N^q, \bm C_{NM}$, $\bm Q_{MN}$, $\bm D_{MK}^d$, $\bm S_{KN}^d$.}
\STATE{Solve the parametric FOM (\ref{eq:1a}) for each $\bm \mu_j \in S_N$ to obtain $u(\bm \mu_j)$.}
\STATE{Construct a RB space $W_N^u = \mbox{span} \{\zeta_j \equiv  u(\bm \mu_j), \ 1 \leq j \leq N\}$.}
\STATE{Construct $\{\psi_m(\bm x)\}_{m=1}^M$ and $\{\bm y_m \in \Omega\}_{m=1}^M$ to interpolate $g(u_N(\bm \mu), \bm \mu)$.}
\STATE{Construct $\{\phi_k^d(\bm x)\}_{k=1}^{K^d}$ and $\{\bm z^d_m \in \Omega\}_{k=1}^{K^d}$ to interpolate $f^d(u_N(\bm \mu),  \bm \mu)$.}
\STATE{Form and store $\bm l_N, \bm l_N^O, \bm A_N^q, \bm C_{NM}$, $\bm Q_{MN}$, $\bm D_{MK}^d$, $\bm S_{KN}^d$.}
\end{algorithmic}
\caption{Offline stage of the Hyperreduction--Galerkin-Newton method.}
\end{algorithm}

\begin{algorithm}
\begin{algorithmic}[1]
\REQUIRE{Parameter point $\bm \mu \in \mathcal{D}$ and initial guess $\bar{\bm \alpha}_N(\bm \mu)$.}
\ENSURE{RB output $s_{N}(\bm \mu)$ and updated coefficients $\bar{\bm{\alpha}}_N(\bm \mu)$.}
\STATE{Linearize (\ref{eq21b}) around $\bar{\bm{\alpha}}_N(\bm \mu)$.}
\STATE{Solve the resulting linear system to obtain $\delta {\bm \alpha}_N(\bm \mu)$.}
\STATE{Update $\bar{\bm \alpha}_N(\bm \mu) = \bar{\bm \alpha}_N(\bm \mu) + \delta {\bm \alpha}_N(\bm \mu)$.}
\STATE{If $\|\delta {\bm \alpha}_N(\bm \mu)\| \le \epsilon$, then calculate $s_{N}(\bm \mu) = (\bm l_N^O)^T \bar{\bm \alpha}_N(\bm \mu)$ and stop.
}
\STATE{Otherwise, go back to Step 1.}
\end{algorithmic}
\caption{Online stage of the Hyperreduction--Galerkin-Newton method.}
\end{algorithm}

\subsection{Newton-Gakerin--Hypereduction method}

This method applies hyperreduction to the linearized system (\ref{eq:6a}) of the Newton-Gakerin method. For each Newton iteration, the method solves the following linear system
\begin{equation}
\left(\sum_{q=1}^Q \Theta^q(\bm \mu) \bm A_N^q + \widetilde{\bm B}_N(\bar{\bm \alpha}_N(\bm \mu))   \right) \delta \bm \alpha_N(\bm \mu) = \tilde{\bm r}_N(\bar{\bm \alpha}_N(\bm \mu)) ,
\label{eq:61a}
\end{equation} 
where
\begin{equation}
\tilde{\bm r}_N(\bar{\bm \alpha}_N(\bm \mu)) = \bm l_N - \tilde{\bm b}_N(\bar{\bm \alpha}_N(\bm \mu)) - \left(\sum_{q=1}^Q \Theta^q(\bm \mu) \bm A_N^q \right) \bar{\bm \alpha}_N(\bm \mu) .
\label{eq:61b}
\end{equation} 
The vector $\tilde{\bm b}_N(\bar{\bm \alpha}_N(\bm \mu))$ is computed by (\ref{eq19b}), and the matrix  $\widetilde{\bm B}_N(\bar{\bm \alpha}_N(\bm \mu)) \in \mathbb{R}^{N \times N}$ has entries
\begin{multline}
\widetilde{B}_{N, ij}(\bar{\bm \alpha}_N(\bm \mu)) = \sum_{m=1}^{M^{g_u}} \beta_m^{g_u}(\bar{\bm \alpha}_N(\bm \mu)) 
 \int_{\Omega} \psi_m^{g_u} \zeta_i \zeta_j \, d  \bm x \, + \\ \sum_{d=1}^D  \sum_{k=1}^{K^{f_u^d}} \gamma^{f_u^d}_k(\bar{\bm \alpha}_N(\bm \mu)) 
 \int_{\Omega} \phi_{k}^{f_u^d} \frac{\partial  \zeta_i}{\partial x_d} \zeta_j \, d  \bm x ,
\end{multline} 
where
\begin{equation}
\label{eq151b}
\sum_{m=1}^{M^{g_u}}  \psi_m^{g_u} (\bm y_k^{g_u}) \beta_m^{g_u}(\bm \mu) = g_u(\bar{u}_N(\bm y_k^{g_u}, \bm \mu),  \bm \mu), \qquad 1 \le k \le M^{g_u} .
\end{equation}
\begin{equation}
\label{eq171b}
\sum_{k=1}^{K^{f_u^d}}  \phi^{f_u^d}_k (\bm z_m^{f_u^d}) \gamma^{f_u^d}_k(\bm \mu) = f_u(\bar{u}_N(\bm z_m^{f_u^d}, \bm \mu),  \bm \mu), \qquad 1 \le m \le K^{f_u^d} .
\end{equation}
Here $\{\psi^{g_u}_m(\bm x)\}_{m=1}^{M^{g_u}}$ and $\{\bm y^{g_u}_m \in \Omega\}_{m=1}^{M^{g_u}}$ are used to define the interpolation of $g_u(u_N(\bm \mu),  \bm \mu)$, while $\{\phi_k^{f_u^d}(\bm x)\}_{k=1}^{K^{f_u^d}}$ and $\{\bm z^{f_u^d}_m \in \Omega\}_{k=1}^{K^{f_u^d}}$ are used to define the interpolation of $f_u^d(u_N(\bm \mu), \bm \mu)$. Thus, the matrix  $\widetilde{\bm B}_N(\bar{\bm \alpha}_N(\bm \mu))$ can be expressed as a sum of parameter-independent matrices as follows
\begin{equation}
\widetilde{\bm B}_{N}(\bar{\bm \alpha}_N(\bm \mu)) = \sum_{m=1}^{M^{g_u}} \beta_m^{g_u}(\bar{\bm \alpha}_N(\bm \mu)) \bm C_m^{g_u} + \sum_{d=1}^D  \sum_{k=1}^{K^{f_u^d}} \gamma^{f_u^d}_k(\bar{\bm \alpha}_N(\bm \mu)) \bm D_k^{f^d_u} 
\label{eq27b}
\end{equation} 
where $C_{mij}^{g_u} = \int_{\Omega} \psi_m^{g_u} \zeta_i \zeta_j \, d  \bm x$ and $D_{kij}^{f^d_u} = \int_{\Omega} \phi_{k}^{f_u^d} \frac{\partial  \zeta_i}{\partial x_d} \zeta_j \, d  \bm x$ are independent of $\bm \mu$.

The offline and online stages of the Galerkin-Newton--Hyperreduction (GN-H) method are summarized in Algorithm 3 and Algorithm 4, respectively. In addition to approximating the nonlinear functions, the method also approximates the partial derivative of the functions with respect to the field variable and thus differs from the H-GN method which approximates the nonlinear functions only. For each Newton iteration during the online stage, the operation counts include $O((M + \sum_{d=1}^D K^{d}) N)$ to form $\tilde{\bm r}_N(\bar{\bm \alpha}_N(\bm \mu))$, $O((Q + M^{g_u} + \sum_{d=1}^D K^{f^d_u}) N^2)$ to form $\widetilde{\bm B}_N(\bar{\bm \alpha}_N(\bm \mu))$, and $O(N^3)$ to solve the reduced system (\ref{eq:61a}). Hence, the computational complexity of the online stage is $O(N^3 + (Q + M^{g_u} + \sum_{d=1}^D K^{f^d_u}) N^2 + (M + \sum_{d=1}^D K^{d}) N)$ per Newton iteration. Specifically, the GN-H method allows 
$M$ and $K^d$  to be chosen large enough to accurately approximate the residual vector without significantly increasing computational cost. When the residual vector is well approximated, resulting in negligible hyperreduction error, the GN-H method can achieve accuracy comparable to the GN method, making it more efficient and accurate than the H-GN method.

\begin{algorithm}
\begin{algorithmic}[1]
\REQUIRE{The parameter sample set $S_N = \{\bm \mu_j, 1 \le j \le N\}$.}
\ENSURE{$\bm l_N, \bm l_N^O, \bm A_N^q, \bm C_{NM}$, $\bm Q_{MN}$, $\bm D_{MK}^d$, $\bm S_{KN}^d$.}
\STATE{Solve the parametric FOM (\ref{eq:1a}) for each $\bm \mu_j \in S_N$ to obtain $u(\bm \mu_j)$.}
\STATE{Construct a RB space $W_N = \mbox{span} \{\zeta_j \equiv  u(\bm \mu_j), \ 1 \leq j \leq N\}$.}
\STATE{Construct $\{\psi_m\}_{m=1}^M$ and $\{\bm y_m \in \Omega\}_{m=1}^M$ to interpolate $g(u_N(\bm \mu),  \bm \mu)$.}
\STATE{Construct $\{\psi_m^{g_u}\}_{m=1}^{M^{g_u}}$ and $\{\bm y_m^{g_u} \in \Omega\}_{m=1}^{M^{g_u}}$ to interpolate $g_u(u_N(\bm \mu), \bm \mu)$.}
\STATE{Construct $\{\phi_k^d\}_{k=1}^{K^d}$ and $\{\bm z^d_m \in \Omega\}_{k=1}^{K^d}$ to interpolate $f^d(u_N(\bm \mu),  \bm \mu)$.}
\STATE{Construct $\{\phi_k^{f^d_u}\}_{k=1}^{K^{f^d_u}}$ and $\{\bm z^{f^d_u}_m \in \Omega\}_{k=1}^{K^{f^d_u}}$ to interpolate $f^d_u(u_N(\bm \mu),  \bm \mu)$.}
\STATE{Form and store $\bm l_N, \bm l_N^O, \bm A_N^q, \bm C_{NM}$, $\bm Q_{MN}$, $\bm D_{MK}^d$, $\bm S_{KN}^d$.}
\end{algorithmic}
\caption{Offline stage of the Galerkin-Newton--Hyperreduction method.}
\end{algorithm}

\begin{algorithm}
\begin{algorithmic}[1]
\REQUIRE{Parameter point $\bm \mu \in \mathcal{D}$ and initial guess $\bar{\bm \alpha}_N(\bm \mu)$.}
\ENSURE{RB output $s_{N}(\bm \mu)$ and updated coefficients $\bar{\bm{\alpha}}_N(\bm \mu)$.}
\STATE{Compute $\tilde{\bm b}_N(\bar{\bm \alpha}_N(\bm \mu))$ from (\ref{eq19b}).}
\STATE{Compute $\tilde{\bm r}_N(\bar{\bm \alpha}_N(\bm \mu))$ from (\ref{eq:61b}).}
\STATE{Solve (\ref{eq151b}) and (\ref{eq171b}) for $\beta_m^{g_u}(\bm \mu)$ and $\gamma_k^{f_u^d}(\bm \mu)$, respectively.}
\STATE{Compute $\widetilde{\bm B}_N(\bar{\bm \alpha}_N(\bm \mu))$ from (\ref{eq27b}).}
\STATE{Solve the linear system (\ref{eq:61a}) to obtain $\delta {\bm \alpha}_N(\bm \mu)$.}
\STATE{Update $\bar{\bm \alpha}_N(\bm \mu) = \bar{\bm \alpha}_N(\bm \mu) + \delta {\bm \alpha}_N(\bm \mu)$.}
\STATE{If $\|\delta {\bm \alpha}_N(\bm \mu)\| \le \epsilon$, then calculate $s_{N}(\bm \mu) = (\bm l_N^O)^T \bar{\bm \alpha}_N(\bm \mu)$ and stop.
}
\STATE{Otherwise, go back to Step 1.}
\end{algorithmic}
\caption{Online stage of the Galerkin-Newton--Hyperreduction method.}
\end{algorithm}

\section{High-Order Empirical Interpolation Methods}

The  empirical interpolation method (EIM) was first introduced in \cite{Barrault2004} for constructing basis functions and interpolation points to approximate parameter-dependent functions, and developing efficient RB approximation of non-affine PDEs. Shortly later, the empirical interpolation method was extended to develop efficient ROMs for nonlinear PDEs \cite{Grepl2007}. Since the pioneer work \cite{Barrault2004, Grepl2007}, the EIM has been widely used to construct efficient ROMs of nonaffine and nonlinear PDEs for  different applications \cite{Grepl2007,Nguyen2007,Galbally2010,Drohmann2012,Manzoni2012,Hesthaven2014,Kramer2019,Hesthaven2022,Chen2021}. Rigorous a posteriori error bounds for the empirical interpolation method is developed by Eftang et al. \cite{Eftang2010b}. Several attempts have been made to extend the EIM in diverse ways. The best-points interpolation method (BPIM) \cite{Nguyen2008a,Nguyen2008d} employs proper orthgogonal decomposition to generate the basis set and least-squares method to compute the interpolation point set.  Generalized empirical interpolation method (GEIM) \cite{Maday2015a,Maday2013} generalizes the EIM concept by replacing the pointwise function evaluations by more general measures defined as linear functionals. 

%The discrete empirical interpolation method (DEIM) \cite{Chaturantabut2010} is a discrete variant of the empirical interpolation method.
%The  empirical interpolation method (EIM) \cite{Barrault2004a,Grepl2007a}  was proposed as an efficient model reduction technique to render the complexity of evaluating the nonlinear terms independent of the dimension of the truth approximation. 

The main idea in the  empirical interpolation method is to replace any nonlinear term with a reduced basis expansion expressed as a linear combination of pre-computed basis functions and parameter-dependent coefficients. The coefficients are determined efficiently by an inexpensive and stable interpolation procedure. In the empirical interpolation method, the basis functions are instances of the nonlinear function at $N$ parameter points in a sample set $S_N$. Therefore, the number of basis functions does not exceed $N$. In order to improve the approximation accuracy, we must increase $N$ at the expense of increasing the offline cost because we need to solve the FOM for all parameter points in $S_N$. ROMs via EIM have been shown to suffer from instabilities in certain situations \cite{Peherstorfer2020}. Adaptation \cite{Eftang2012b} and localization \cite{Peherstorfer2014} of the low-dimensional subspaces have been proposed to improve the stability of ROMs via empirical interpolation. Oversampling uses more interpolation points than basis functions so that the nonlinear terms are approximated via least-squares regression rather than via interpolation \cite{Peherstorfer2020}.  Oversampling methods such as  gappy POD \cite{everson95karhunenloeve,willcox06:_gappy}, missing point estimation \cite{Astrid2008,Zimmermann2016}, Gauss-Newton with approximated tensors \cite{Carlberg2013}, and GEIM \cite{Argaud2017} can provide more stable and accurate approximations than empirical interpolation especially when the samples are perturbed due to noise turbulence, and numerical inaccuracies. While the number of interpolation points can be greater than $N$, the number of basis functions still does not exceed $N$. Hence, oversampling shows a limited improvement over the interpolation approach. 

%Our approach is based on the Taylor expansion. expand the dimension of the interpolating subspace and increase the number of interpolation points. In particular, we

 The first-order empirical interpolation method (FOEIM), introduced in \cite{Nguyen2023d,Nguyen2024}, enhances the EIM by utilizing the partial derivatives of a parametrized nonlinear function to generate additional basis functions and interpolation points. This method can construct up to $N^2$ basis functions and interpolation points from a given sample set $S_N$, greatly increasing the approximation power without requiring extra FOM solutions. FOEIM significantly improves the accuracy of hyper-reduced ROMs, especially in capturing complex nonlinear behaviors, while maintaining computational efficiency. 

In this section, we introduce a class of high-order empirical interpolation methods to generate up to $N^{p+1}$  functions and interpolation points by leveraging $p^{\rm th}$-order partial derivatives of the parametrized nonlinear function. As the number of functions scales exponentially with $p$, we use proper orthogonal decomposition to compute the basis functions. We develop effective {\em a posteriori} estimator to quantify the interpolation errors and construct the parameter sample $S_N$ via greedy sampling. Using high-order partial derivatives, high-order EIM provides a more flexible and accurate alternative to the original EIM and FOEIM.

 %We introduce a class of high-order empirical interpolation methods by using high-order derivatives. Th

 \subsection{Interpolation procedure}

We aim to interpolate the nonlinear function $g(u(\bm x, \bm \mu), \bm \mu)$ using a set of basis functions $\Psi_M = \mbox{span} \{\psi_m(\bm x), 1 \le m \le M\}$ and a set of interpolation points $T_M = \{{\bm y}_1, \ldots, {\bm y}_M\}$. The interpolant $g_M(\bm x, \bm \mu)$ is given by 
\begin{equation}
\label{eq1w}
g_M(\bm x, \bm \mu) = \sum_{m=1}^M \beta_{M,m}(\bm \mu) \psi_m(\bm x)   ,
\end{equation}
where the coefficients $\beta_{M,m}(\bm \mu), 1 \le m \le M,$ are found as the solution of the following linear system 
\begin{equation}
\label{eq2w}
\sum_{m=1}^M  \psi_m({\bm y}_k)   \beta_{M,m}(\bm \mu) = g(u({\bm y}_k, \bm \mu),  \bm \mu), \quad 1 \le k \le M .
\end{equation}
It is convenient to compute the coefficient vector $\bm \beta_M(\bm \mu)$ as follows
\begin{equation}
\label{eqcoeff}
\bm \beta_M(\bm \mu) = \bm B^{-1}_M \bm b_M(\bm \mu),
\end{equation}
where $\bm B_M \in \mathbb{R}^{M \times M}$ has entries $B_{M,km} =  \psi_m({\bm y}_k)$ and $\bm b_M(\bm \mu) \in \mathbb{R}^M$ has entries $b_{M,k}(\bm \mu) =  g(u({\bm y}_k, \bm \mu),  \bm \mu)$. The interpolation error is defined as
\begin{equation}
\label{eq4w}
\varepsilon_M(\bm \mu) = \|g(u(\bm x, \bm \mu), \bm \mu) - g_M(\bm x, \bm \mu)\|_{L^\infty(\Omega)} .     
\end{equation}
The complexity of computing the coefficient vector  $\bm \beta_M(\bm \mu)$ in (\ref{eqcoeff}) for any given $\bm \mu$ is $O({M^2})$  because the matrix $\bm B_M^{-1}$ is pre-computed and stored.

The approximation accuracy depends critically on both the  subspace $\Psi_M$ and the interpolation point set $T_M$.   In what follows, we describe high-order empirical interpolation methods for constructing $\Psi_M$ and $T_M$.

\subsection{Basis functions}

For a given sample set $S_N = \{\bm \mu_n, 1 \le n \le N\}$, we introduce the following  spaces 
\begin{subequations}
\label{EIMspaces}
\begin{alignat}{2}
W_N & = \mbox{span} \{\zeta_n(\bm x) \equiv u(\bm x, \bm \mu_n), 1 \le n \le N\}, \\
V_N^g  & = \mbox{span} \{\xi_n(\bm x) \equiv g(\zeta_n(\bm x), \bm \mu_n), 1 \le n \le N\} .
\end{alignat}
\end{subequations}
These  spaces are constructed by solving the underlying FOM (\ref{eq:1a}) for each $\bm \mu_n \in S_N$. The first-order empirical interpolation method is based on the first-order Taylor expansion of $g(u, \bm \mu)$ at $(\zeta, \bm \eta)$, which is defined as follows
\begin{equation}
G((u, \bm \mu), (\zeta, \bm \eta)) = g(\zeta, \bm \eta) + \frac{\partial g(\zeta, \bm \eta)}{\partial u} (u - \zeta) + \frac{\partial g(\zeta, \bm \eta)}{\partial \bm \mu} \cdot (\bm \mu - \bm \eta)  .
\end{equation}
Taking $(u, \bm \mu) = (\zeta_m, \bm \mu_m)$ and $(\zeta, \bm \eta) = (\zeta_n, \bm \mu_n)$, where $(\zeta_m,\zeta_n)$ are any pair of two functions in $W_N$ and $(\bm \mu_m,\bm \mu_n)$ are any pair of two parameter points in $S_N$, we arrive at
\begin{equation}
\label{taylor}
G((\zeta_m, \bm \mu_m), (\zeta_n, \bm \mu_n)) = g(\zeta_n,  \bm \mu_n) + \frac{\partial g(\zeta_n, \bm \mu_n)}{\partial u} (\zeta_m - \zeta_n) + \frac{\partial g(\zeta_n, \bm \mu_n)}{ \partial \bm \mu} \cdot (\bm \mu_m - \bm \mu_n)  .
\end{equation}
for $m,n = 1,\ldots, N$. These functions define the following Lagrange-Taylor space
\begin{equation}
\label{LTspace}
V_{N^2}^{g}   = \mbox{span} \{ \varrho_k \equiv G((\zeta_m, \bm \mu_m), (\zeta_n, \bm \mu_n)), \ k = m + N(n-1), \ 1 \le m,n \le N\}  .
\end{equation}
Note that we have $V_N^g \subset V^g_{N^2}$. Furthermore, $V_{N^2}^{g}$ can be significantly  richer than $V_N^g$ as it contains considerably more basis functions.

In addition to first-order derivatives, we can incorporate second-order partial derivatives to further enhance the accuracy and flexibility of the  method. By utilizing second-order derivatives, we can generate even more refined basis functions and interpolation points to provide a more accurate approximation of the nonlinear function. Specifically, incorporating second-order partial derivatives yields the following functions     
\begin{multline}
\label{taylor2}
G((\zeta_m, \bm \mu_m), (\zeta_n, \bm \mu_n), (\zeta_k, \bm \mu_k)) = g(\zeta_n,  \bm \mu_n) + \frac{\partial g(\zeta_n, \bm \mu_n)}{\partial u} (\zeta_m - \zeta_n) \, + \\  \frac{\partial g(\zeta_n, \bm \mu_n)}{ \partial \bm \mu} \cdot (\bm \mu_m - \bm \mu_n) + \frac{1}{2} \frac{\partial^2 g(\zeta_n, \bm \mu_n)}{\partial u^2} (\zeta_k - \zeta_n) (\zeta_m - \zeta_n) \, + \\ 
\frac{1}{2} (\bm \mu_k - \bm \mu_n)^T \frac{\partial^2 g(\zeta_n, \bm \mu_n)}{ \partial \bm \mu^2}  (\bm \mu_m - \bm \mu_n) 
\end{multline}
for $1 \le k,m,n \le N$. Let $V^g_{N^3}$ be the space spanned by those functions. Then we have $V_N^g \subset V_{N^2}^g \subset V_{N^3}^g$. Incorporating second-order partial derivatives is particularly beneficial when dealing with systems where higher-order nonlinearities play a significant roles. It can lead to significant improvements in the stability and accuracy of hyper-reduced ROMs, especially when applied to highly nonlinear or stiff systems. Furthermore, the inclusion of second-order information does not require additional FOM evaluations, making it computationally efficient. The method relies solely on the partial derivatives of the already available snapshots, maintaining the advantage of reducing the offline computational cost.

% If the function $g$ is explicitly dependent on both $u$ and $\bm \mu$, then  we define the LT space as $W_{N^2}^{g}   = \mbox{span} \{\eta_k, 1 \le k \le N^2\}$, where 
% \begin{equation}
% \label{taylor2}
% \eta_k = g(\zeta_n, \bm \mu_n) + \frac{\partial g(\zeta_n, \bm \mu_n)}{\partial u} (\zeta_m - \zeta_n) + \frac{\partial g(\zeta_n, \bm \mu_n)}{ \partial \bm \mu} \cdot (\bm \mu_m - \bm \mu_n) . 
% \end{equation}
% for $k = m + N(n-1)$ and $m,n = 1,\ldots, N$. Here $\bm \mu_n, 1 \le n \le N,$ are the parameter points in the sample set $S_N$. 

%As the method can generate a large number of  functions ($N^2$ functions if using only the first-order derivatives and $N^3$ if using the first-order and second-order derivatives), it is necessary to compress them. Without compression, the large basis set can lead to increased computational cost and reduced efficiency during the online stage. Here we use the proper orthogonal decomposition (POD) to generate an orthogonal basis set from a snapshot set $\{\varrho_k\}_{k=1}^K$, where $K$ can be $N^2$ or $N^3$. 

As the method can generate a large number of functions (up to $N^2$  using first-order derivatives and 
$N^3$ when including second-order derivatives), it becomes necessary to compress this large set of functions to avoid excessive computational cost. To address this, we employ proper orthogonal decomposition (POD), which generates an optimal reduced orthogonal basis from a set of snapshots $\{\varrho_k\}_{k=1}^K$, where $K$ represents the total number of basis functions and can scale to $N^2$ or $N^3$. By applying POD, we can limit the dimensionality of the basis set to a manageable size by retaining the dominant modes, thus ensuring that the online stage can be executed rapidly without compromising the ROM accuracy. 

The POD determines the basis set to minimize the following mean squared error:
\begin{equation}
\min \sum_{i=1}^{K}  \Big\|\varrho_{i} -  \frac{(\varrho_{i}, \varphi)}{\|\varphi\|^2} \varphi \Big\|^2 \ ,
\label{eq3a:8}
\end{equation} 
% This minimization problem is equivalent to the maximization problem~
% \begin{equation}
% \max   \sum_{i=1}^{N^2}   |\left(\varphi, \eta_{i}\right)|^2
% \label{eq3a:4}
% \end{equation}
subject to the constraints $\|\varphi\|^2 = 1$. It is shown in \cite{Holmes2012} (see Chapter 3) that the problem~(\ref{eq3a:8}) is equivalent to solving the eigenfunction equation 
\begin{equation}
\frac{1}{K} \sum_{i=1}^{K}   (\varrho_{i}, \varphi) \varrho_{i} = \lambda \, \varphi .
\label{eq3a:5}
\end{equation}
%As the optimal basis functions are given by the eigenfunctions of the eigenfunction equation, they are  called empirical eigenfunctions or POD modes. 
The method of snapshots~\cite{sirovich87:_turbul_dynam_coher_struc_part} expresses a typical eigenfunction $\varphi$ as a linear combination of the snapshots 
\begin{equation}
\varphi =  \sum_{j=1}^{K}   a_{j} \varrho_{j} \ .
\label{eq3a:6}
\end{equation}
Inserting (\ref{eq3a:6}) into~(\ref{eq3a:5}), we  obtain the following eigenvalue problem
\begin{equation}
\bm C \bm a = \lambda \bm a \ ,
\label{eq3a:7}
\end{equation}
where $\bm C \in \mathbb{R}^{K \times K}$ is known as the correlation matrix with entries $C_{ii'}  = \frac{1}{K} \left(\varrho_{i},\varrho_{j} \right), 1 \le i, j \le K$. The eigenproblem~(\ref{eq3a:7}) can then be solved for the first $M$ eigenvalues and eigenvectors from which the POD basis functions $\varphi_m, 1 \le m \le M,$ are constructed by~(\ref{eq3a:6}). We then introduce the following space
\begin{equation}
\Phi_M := \mbox{span} \{\phi_m = \sqrt{\lambda}_m \varphi_m, \quad 1 \le m \le  M\} .
\label{eq3a:7qq}
\end{equation}
Note that the eigenvalues follow the descending order $\lambda_1 \ge \lambda_2 \ge \ldots \ge \lambda_M$. Typically, $M$ is chosen significantly less than $K$. This dimensionality reduction drastically reduces the number of basis functions required to achieve accurate approximation, thereby allowing for faster computations and reduced memory usage during the online stage.

%The integer $M$ is chosen as the smallest integer such that $\sum_{k=K+1}^{N^2} \lambda_k / \sum_{n=1}^{N^2} \lambda_n \le \epsilon$ for some $\epsilon$ small. 
%The sum of all of the eigenvalues represents the total energy content of the snapshot set. If we want to capture at least 99\% of the energy, then we must choose the smallest integer $K$ such that $\sum_{k=1}^{K} \lambda_k / \sum_{n=1}^{N^2} \lambda_n \ge 0.99$.

%POD captures the most dominant modes in the LT space to enable accurate representation with fewer basis functions. In particular, the LT-POD space $\Phi_K^{g}$ can accurately approximate any function in $W_{N^2}^{g}$ even though the dimension $K$ may be significantly less than $N^2$. We will apply the EIM procedure to $\Phi_K^{g}$ instead of $W_{N^2}^{g}$. Hence, POD reduces the computational cost of the EIM procedure by a factor of $N^2/K$, while generating optimal basis functions. %To simplify notation, we shall drop the superscript $\partial^L$ in the remainder of this chapter.  

\subsection{Interpolation points}

We apply the EIM procedure \cite{Maday2008c} to the space $\Phi_M$ defined in (\ref{eq3a:7qq}) to obtain interpolation points. First, we find
\begin{equation}
\label{eq23w}
j_1 = \arg \max_{1 \le j \le M} \|\phi_j  \|_{L^\infty(\Omega)},  
\end{equation}
and set
\begin{equation}
{\bm y}_1 = \arg \sup_{\bm x \in \Omega} |\phi_{j_1}(\bm x)|, \qquad \psi_1(\bm x) = \phi_{j_1}(\bm x)/\phi_{j_1}({\bm y}_1).      
\end{equation}
For $m = 2, \ldots, M$, we solve the linear systems
\begin{equation}
\sum_{i=1}^{m-1}  \psi_i({\bm y}_j)   \sigma_{li} = \phi_l({\bm y}_j), \quad 1 \le j \le m-1 , 1 \le l \le M,
\end{equation}
we then find
\begin{equation}
\label{argmax}
j_{m} = \arg \max_{1 \le l \le M} \| \phi_l(\bm x) - \sum_{i=1}^{m-1}  \sigma_{li} \psi_i(\bm x)\|_{L^\infty(\Omega)}, 
\end{equation}
and set
\begin{equation}
{\bm y}_m = \arg \sup_{\bm x \in \Omega} |r_M(\bm x)|, \qquad \psi_m(\bm x) = r_m(\bm x)/r_m({\bm y}_m)    ,   
\end{equation}
where the residual function $r_m(\bm x)$ is given by
\begin{equation}
\label{eq28w}
r_m(\bm x) = \phi_{j_m}(\bm x) - \sum_{i=1}^{m-1}  \sigma_{j_m i} \psi_i(\bm x) . 
\end{equation}
% finally, we define
% \begin{equation}
% \label{eq28wq}
% \Psi_M^{ g} = \mbox{span} \{\psi_1^{ g}, \ldots, \psi_M^{ g}\}, \qquad 
% T_M^{ g} = \{{\bm x}_1^{ g}, \ldots, {\bm x}_1^{ g}\} .
% \end{equation}
In practice, the supremum $\sup_{\bm x \in \Omega} |r_M(\bm x)|$ is computed on the set of quadrature points on all elements in the mesh. In other words, the interpolation points $\{\bm y_m\}_{m=1}^M$ are selected from the quadrature points.

For any given parameter sample set $S_N$, the high-order EIM method constructs $M$ interpolation points $\{\bm y_m\}_{m=1}^M$ and $M$ basis functions $\{\psi_m\}_{m=1}^M$ by leveraging partial derivatives. In contrast, the original EIM only generates $N$ interpolation points and basis functions. If we wish to generate more than  $N$ interpolation points and basis functions with the original EIM, we must expand the parameter sample set and solve the underlying FOM for additional solutions. However, doing so increases the computational cost in both the offline and online stages.   By allowing 
$M > N$, the high-order EIM enhances ROM stability and accuracy without additional parameter samples to provide cost-effective construction of ROMs. The computational cost of the online stage scales linearly with $M$, making the method both efficient and accurate.

\subsection{Stability of the high-order empirical interpolation}

%The high-order EIM has all the desirable properties of the  EIM. The first-order EIM is well-defined in the sense that the basis functions are linearly independent and the matrix $\bm B_M^{ g}$ with entries $B_{M, ij}^{ g} = \psi_j^{ g}({\bm x}_i^{ g}), 1 \le i,j, \le M,$ is invertible. 
Let $\Psi_{M} = {\rm span} \: \{ \psi_1, \dots,\psi_{M} \}$ and $T_M = \{\bm y_1, \ldots, \bm y_M\}$. For any given function $v(\bm x)$, we define its interpolant as follows
\begin{equation}
\mathcal{I}(\Psi_M, T_M)[v(\bm x)] = \sum_{m=1}^M a_{m}\psi_m(\bm x)  ,
\end{equation}
where the coefficients $a_{m}, 1 \le m \le M,$ are given by
\begin{equation}
\label{eqcoeff1q}
\sum_{m=1}^M \psi_m(\bm y_k) a_{m} = v(\bm y_k), \quad k = 1,\ldots, M.
\end{equation}
This system of equations can be written in the matrix form as
\begin{equation}
\label{eqcoeff2q}
\bm a_M = \left( \bm B_M \right)^{-1}\bm v_M .
\end{equation}
Here $\bm B_M$ has entries $B_{M,mk} = \psi_m(\bm y_k)$ and $\bm v_M$ has entries $v_{M,m} =  v(\bm y_m)$ for $1 \le m,k \le M$.

%In what follows, we denote the interpolant of any  function $v$ by $\mathcal{I}_M^{ g}[v]$ to simplify the notation.  We obtain an intermediate result:

\begin{lemma}
Assume that $\Psi_{M}$ is of dimension $M$ and that $\bm B_{M}$ is invertible, then we have $\mathcal{I}(\Psi_M, T_M)[v] = v$ for any $v \in \Psi_{M}$. In other words, the interpolation is exact for all v in $\Psi_{M}$.
\end{lemma}
\begin{proof}
For $v \in \Psi_{M}$, which can be expressed as $v(\bm x) = \sum_{j=1}^{M} c_{M,j} \psi_j(\bm x)$, we consider $\bm x = {\bm y}_m$ to obtain $v({\bm y}_m) = \sum_{j=1}^{M} c_{M,j} \psi_j({\bm y}_m), 1 \leq m \leq M$. It thus follows from the invertibility of $\bm B_{M}$ that $\bm c_{M} = \bm a_{M}$. Hence, we have $\mathcal{I}(\Psi_M, T_M)[v] = v$. 
\end{proof}

We can then prove the following theorem:
\begin{thm}
\label{thm1}
The space $\Psi_M$ is of dimension $M$. In addition, the matrix $\bm B_M$ is lower triangular with unity diagonal.   
\end{thm}
\begin{proof}
We shall proceed by induction. Clearly, $\Psi_1 = {\rm span} \: \{ \psi_1  \}$ is of dimension 1 and the matrix $\bm B_1 = 1$ is invertible.  Next we assume that $\Psi_{M-1} = {\rm span} \: \{ \psi_1 , \dots, \psi_{M} \}$ is of dimension $M-1$ and the matrix $\bm B_{M-1}$ is invertible. We must then prove ({\it i}) $\Psi_M  = \mbox{span}\{\psi_1, \ldots, \psi_M \}$ is of dimension $M$ and ({\it ii}) the matrix $\bm B_M$ is invertible. To prove ({\it i}), we note from our ``arg max'' construction (\ref{argmax}) that $\|r_M(\bm x)\|_{L^\infty(\Omega)} > 0$. Hence, if $\dim(\Psi_M) \neq M$, we have $\phi_{j_M} \in \Psi_{M}$ and thus $\|r_M(\bm x)\|_{L^\infty(\Omega)} = 0$ by Lemma 1; however, the latter contradicts $\|r_M(\bm x)\|_{L^\infty(\Omega)} > 0$.  To prove ({\em ii\/}), we just note from the
  construction procedure (\ref{eq23w})-(\ref{eq28w}) that $B_{{M},{i \: j}} = r_j({\bm y}_i ) /r_j ({\bm y}_j) = 0$
  for $i < j$; that $B_{{M},{i \: j}} = r_j({\bm y}_i ) /r_j ({\bm y}_j ) = 1$ for $i =
  j$; and that $\left|B_{{M},{i \: j}} \right| = \left|r_j({\bm y}_i ) /r_j ({\bm y}_j ) \right| \leq 1$ for $i > j$ since ${\bm y}_j  = \arg  \:
  \max_{x \in \Omega} |r_j (\bm x)|, 1 \leq j \leq M$. Hence, $\bm B_{M}$ is
  lower triangular with unity diagonal. Hence, $B^M$ is invertible. 
\end{proof}

This theorem implies that the high-order EIM yields unique interpolation points and linearly independent basis functions.

\subsection{Convergence of the high-order empirical interpolation}

The convergence analysis of the interpolation procedure involves the Lebesgue constant as follows
\begin{thm}
The interpolation error $\varepsilon_M(\bm \mu) \equiv \|g(u(\bm x, \bm \mu), \bm \mu) -  g_M(\bm x, \bm \mu) \|_{L^\infty(\Omega)}$ is bounded by
\begin{equation}
\label{EIMbound}
\varepsilon_M(\bm \mu) \le (1 + \Lambda_M)  \inf_{v \in \Psi_M}  \|g(u(\bm x, \bm \mu), \bm \mu) - v \|_{L^\infty(\Omega)}, 
\end{equation}
where $\Lambda_M$ is the Lebesgue constant 
\begin{equation}
\Lambda_M = \sup_{\bm x \in \Omega} \sum_{j=1}^M \left|\sum_{m=1}^M \psi_m(\bm x) [\bm B_M]^{-1}_{mk} \right| .    
\end{equation}
And the Lebesgue constant satisfies $\Lambda_M \le 2^M-1$. 
\end{thm}
\noindent

This result has been proven in \cite{Barrault2004}. The last term in the right hand side of the above inequality is known as the best approximation error. Although the upper bound on the Lebesgue constant is very pessimistic, it can be realized in some extreme cases \cite{Maday2008c}.

For the parametric manifold $\mathcal{G} := \{g(u(\bm x, \bm \mu)) : \bm \mu \in \mathcal{D}\}$, the best approximation of an element $q \in \mathcal{G}$ in some finite dimensional space $\mathcal{X}_n$ of dimension $n$ is given by the orthogonal projection onto $\mathcal{X}_n$, namely, $q^* = \arg \inf_{p \in \mathcal{X}_n} \|q - p \|_{L^\infty(\Omega)}$. In many cases, the evaluation of the best approximation may be costly and the knowledge of $q$ over the entire domain $\Omega$ is required. Thus,  interpolation is referred to as an inexpensive surrogate to the best approximation. The  $n$-width for an interpolation process is defined by
\begin{equation}
\widehat{d}_n(\mathcal{G}) = \inf_{\mathcal{X}_n} \sup_{q \in \mathcal{G}} \|q - \mathcal{I}(\mathcal{X}_n, \mathcal{T}_n) [q] \|_{L^\infty(\Omega)} ,     
\end{equation}
where $\mathcal{I}(\mathcal{X}_n, \mathcal{T}_n)[q]$ denotes an interpolant of $q$ in the linear subspace $\mathcal{X}_n$ using $n$ interpolation points in $\mathcal{T}_n$. The interpolation $n$-width  measures the extent to which $\mathcal{G}$ may be interpolated by the interpolation procedure $\mathcal{I}(\mathcal{X}_n, \mathcal{T}_n)$. The interpolation $n$-width $\widehat{d}_n(\mathcal{G})$ is an upper bound of the Kolmogorov $n$-width 
\begin{equation}
d_n(\mathcal{G}) = \inf_{\mathcal{X}_n} \sup_{q \in \mathcal{G}} \inf_{p \in \mathcal{X}_n} \|q - p \|_{L^\infty(\Omega)} .     
\end{equation}
If $\widehat{d}_n(\mathcal{G})$ converges to zero as $n$ goes to infinity as fast as $d_n(\mathcal{G})$, then the interpolation procedure $\mathcal{I}_n(\mathcal{X}_n, \mathcal{T}_n)$ is stable and accurate. 

%Since $\Psi_M^{ g}$ converges rapidly to $\Phi_K^{g}$ as $M$ tends to $K$, the first-order EIM can yield good approximation spaces for the parametric manifold $\mathcal{G}$.

The interpolation $n$-width raises two important questions: Is there a constructive optimal selection for the interpolation points? Is there a constructive optimal construction of the approximation subspaces? The high-order EIM provides a positive answer to the first question by generating a unique set of interpolation points that yield a stable and unique interpolant. The high-order EIM  also provides a positive answer to the second question by using the partial derivatives to construct good approximation spaces.  Indeed, we note from the high-order EIM procedure that
\begin{equation}
\| q - \mathcal{I}(\Psi_m, T_m)[q] \|_{L^\infty(\Omega)} \le \| \phi_{j_{m+1}} - \mathcal{I}(\Psi_m, T_m)[\phi_{j_{m+1}}]\|_{L^\infty(\Omega)} , \quad \forall q \in  \Phi_M .
\end{equation}
This last quantity is one of the outputs of the first-order EIP and plays the role of {\em a priori} error estimate. The convergence of $\| \phi_{j_{m+1}} - \mathcal{I}(\Psi_m, T_m)[\phi_{j_{m+1}}]\|_{L^\infty(\Omega)}$ as $m$ increases can give a sense of the convergence of the interpolation error.

%The largest possible distance between $\mathcal{G}$ and $\mathcal{X}_n$ is defined as 
% \begin{equation}
% d (\mathcal{G}, \mathcal{X}_n) := \sup_{q \in \mathcal{G}} \inf_{\xi \in \mathcal{X}_n} \|q - \xi \|_{L^\infty(\Omega)} .   
% \end{equation} 
% Since $W_N^g \subset \Phi_K^{ g}$, we have $d (\mathcal{G}, \Phi_K^{ g}) \le d (\mathcal{G}, W_N^g)$. Furthermore, 

\subsection{Error estimate of the high-order empirical interpolation}

The convergence analysis of the high-order empirical interpolation provides an estimate for the number of interpolation points needed to achieve a specific accuracy in the offline stage. However, it does not provide an estimate of the interpolation error of the interpolant for any given parameter $\bm \mu$ in the online stage. The following error estimate has been obtained in \cite{Barrault2004}. 
\begin{prop}
If $g(u(\bm x, \bm \mu)) \in \Psi_{M+1}$, then the interpolation error $\varepsilon_M(\bm \mu) \equiv \|g(u(\bm x, \bm \mu),\bm \mu) -  g_M(\bm x, \bm \mu) \|_{L^\infty(\Omega)}$ is bounded by
\begin{equation}
\label{EIMbound2}
\varepsilon_M(\bm \mu) \le \hat{\varepsilon}_M (\bm \mu) \equiv  | g (u({\bm y}_{M+1}, \bm \mu), \bm \mu) - g_M 
({\bm y}_{M+1}, \bm \mu )| .
\end{equation}
\end{prop}
Of course, in general, $g(u(\bm x, \bm \mu), \bm \mu) \notin \Psi_{M+1}$, and  hence the error estimator $\hat{\varepsilon}_M (\bm \mu)$ is not quite an upper bound. Indeed, $\hat{\varepsilon}_M (\bm \mu)$ must be a lower bound of the interpolation error $\varepsilon_M(\bm \mu)$ due to the fact $| g (u({\bm y}_{M+1}, \bm \mu), \bm \mu) - g_M ({\bm y}_{M+1}, \bm \mu )| 
 \le \|g(u(\bm x, \bm \mu), \bm \mu) -  g_M(\bm x, \bm \mu) \|_{L^\infty(\Omega)}$. We extend the above result to improve the error estimate as follows.

\begin{thm}
\label{thm3}
If $g(u(\bm x, \bm \mu),\bm \mu) \in \Psi_{M+P}$ for $P \in \mathbb{N}_{+}$, then the interpolation error $\varepsilon_M(\bm \mu) \equiv \|g(u(\bm x, \bm \mu), \bm \mu) -  g_M(\bm x, \bm \mu) \|_{L^\infty(\Omega)}$ is bounded by
\begin{equation}
\label{EIMbound3}
\varepsilon_M(\bm \mu) \le \hat{\varepsilon}_{M,P} (\bm \mu) \equiv  \sum_{j=1}^P | e_j(\bm \mu) | ,
\end{equation}
where $e_j(\bm \mu), 1 \le j \le P,$ solve the following linear system
\begin{equation}
\sum_{j=1}^{P} \psi_{M+j}({\bm y}_{M+i}) e_j(\bm \mu)  =   g (u({\bm y}_{M+i}, \bm \mu), \bm \mu) - g_M 
({\bm y}_{M+i}, \bm \mu ), \quad 1 \le i \le P .
\end{equation}
\end{thm}
\begin{proof}
Since by assumption $g(u(\bm x, \bm \mu), \bm \mu) \in \Psi_{M+P}$, we have 
$$g(u(\bm x, \bm \mu), \bm \mu) - g_M(\bm x, \bm \mu)  =
\sum^{M+P}_{m = 1} \kappa_m(\bm \mu) \: \psi_m (\bm x) , $$  
which yields 
$$\sum^{M+P}_{m = 1}   \psi_m  ({\bm y}_i) \: \kappa_m(\bm \mu) = g (u({\bm y}_i, \bm \mu), \bm \mu) -g_M ({\bm y}_i, \bm  \mu),  \quad 1 \leq i \leq M + P .$$
Since $g (u ({\bm y}_i, \bm  \mu), \bm \mu) -g_M ({\bm y}_i, \bm  \mu) = 0$, $1 \leq i \leq M$ and the matrix $\psi_m  ({\bm y}_i)$ is lower triangular with unity diagonal, we have $\kappa_m(\bm \mu) = 0$, $1 \leq m \leq M$. Therefore, the above system reduces to the following  system  
$$\sum^{P}_{j = 1}   \psi_{M+j} ({\bm y}_{M+i}) \: \kappa_{M+j}(\bm \mu) = g (u({\bm y}_{M+i}, \bm \mu), \bm \mu) -g_M ({\bm y}_{M+i}, \bm  \mu),  \quad 1 \leq i \leq P. $$
It follows from Theorem \ref{thm1} that $e_j(\bm \mu) = \kappa_{M+j}(\bm \mu), 1 \le j \le P,$ and thus we obtain
$$g(u(\bm x, \bm \mu), \bm \mu) - g_M(\bm x, \bm \mu)  =
\sum^{P}_{j = 1} e_j(\bm \mu) \: \psi_{M+j} (\bm x) . $$  
The desired result directly follows from taking the ${L^{\infty}(\Omega)}$ norm on both sides, using the triangle inequality, and $\| \psi_{M+j} (\bm x) \|_{L^{\infty}(\Omega)} = 1, 1 \le j \le P$. 
\end{proof}

The operation count of evaluating the error estimator (\ref{EIMbound3}) is only $O(P^2)$. Hence, the error estimator is very inexpensive to evaluate. Error estimation plays a critical role in ensuring the accuracy of the interpolation and guiding the selection of parameter points via greedy sampling. The error estimate  provides a form of certification, allowing us to control and balance the trade-off between computational efficiency and accuracy. Furthermore, parameter points that exhibit the highest error estimates are selected for further FOM evaluations, thus enriching the RB space in the low accuracy regions and ensuring rapid and reliable convergence.

% greedy sampling [refs] is often employed to select parameter samples based on a posteriori error estimates. It aims to accurately capture the solution manifold by ensuring that the largest error estimate over a fine parameter sample is less than a specified tolerance. Greedy sampling iteratively refines the snapshot set by identifying the most informative parameter points and computing the corresponding FOM solutions. However, the greedy sampling  raises an intriguing question: for a given snapshot set, can the quality of the snapshot set be enhanced without requiring additional FOM solutions?

\subsection{Greedy sampling}

Let $\Xi^{\rm train} = \{\bm \mu_j^{\rm train} , 1 \le j \le N^{\rm train}\}$ be a suitably fine training set of $N^{\rm train}$ parameter points.  We assume that we are given an initial parameter set $S_N = \{\bm \mu_1, \ldots, \bm \mu_N\}$, where $N$ is typically small. The greedy sampling approach repeatedly finds the next parameter point as $\bm \mu_{N+1} = \max_{\bm \mu \in \Xi^{\rm train}} \hat{\varepsilon}_{M,P} (\bm \mu)$ and appends $\bm \mu_{N+1}$ to $S_N$ until  $\hat{\varepsilon}_{M,P} (\bm \mu_{N+1})$ is less than a specified tolerance. Here $\hat{\varepsilon}_{M,P} (\bm \mu)$ is the error estimate defined in Theorem \ref{thm3}. This error estimate requires us to evaluate $u(\bm y_i, \bm \mu), 1 \le i \le M+P$, for all $\bm \mu \in \Xi^{\rm train}$.

For the model reduction methods described in Section 2, we actually aim to approximate $g(u_N(\bm x, \bm \mu),  \bm \mu)$ with the following function 
\begin{equation}
\label{eq1wq}
g_M(\bm x, \bm \mu) = \sum_{m=1}^M \beta_{M,m}(\bm \mu) \psi_m(\bm x)   ,
\end{equation}
where the coefficients $\beta_{M,m}(\bm \mu), 1 \le m \le M,$ are found as the solution of the following linear system 
\begin{equation}
\label{eq2wq}
\sum_{m=1}^M  \psi_m({\bm y}_k)   \beta_{M,m}(\bm \mu) = g(u_N({\bm y}_k, \bm \mu), \bm \mu), \quad 1 \le k \le M .
\end{equation}
The error estimate for $\|g(u_N(\bm x, \bm \mu), \bm \mu) -  g_M(\bm x, \bm \mu) \|_{L^\infty(\Omega)}$ is given by
\begin{equation}
\label{EIMbound3w}
\hat{\varepsilon}_{M,P} (\bm \mu) \equiv  \sum_{j=1}^P | e_j(\bm \mu) | ,
\end{equation}
where $e_j(\bm \mu), 1 \le j \le P,$ solve the following linear system
\begin{equation}
\sum_{j=1}^{P} \psi_{M+j}({\bm y}_{M+i}) e_j(\bm \mu)  =   g (u_N({\bm y}_{M+i}, \bm \mu),  \bm \mu) - g_M 
({\bm y}_{M+i}, \bm \mu ), 
\label{eq2wqq}
\end{equation}
for $1 \le i \le P$. Thus, the greedy sampling require us to evaluate $u_N(\bm y_i, \bm \mu) = \sum_{n=1}^N \alpha_n(\bm \mu) \zeta_n(\bm y_i), 1 \le i \le M+P,$ for all $\bm \mu \in \Xi^{\rm train}$, where the RB coefficients $\alpha_n(\bm \mu), 1 \le n \le N,$ are computed by using either the H-GN method or the GN-H method. 

The greedy sampling for the GN-H method is summarized in Algorithm 5. Here we compute the error estimates for both  $g(u_N(\bm \mu), \bm \mu)$ and $\bm f(u_N(\bm \mu), \bm \mu)$, and define $\hat{\varepsilon}_{M,P} (\bm \mu)$ as the maximum value among all error estimates. To accurately calculate these error estimates, both $M$ and $P$ need to be specified, where $M$ represents the number of basis functions and interpolation points used for approximation, and $P$ is the number of additional basis functions and interpolation points used for the error estimation. Typically, we set $P=N$, and choose $M$ as a multiple of $N$ to ensure sufficient interpolation accuracy while maintaining computational efficiency. By selecting 
$M$ as a multiple of $N$, we enhance the resolution of the hyperreduction process without significantly increasing the cost.

With the use of high-order empirical interpolation methods, we can generate up to $N^3$ basis functions to obtain accurate approximations of the nonlinear terms. This expanded basis set ensures that the proposed approach can handle more complex nonlinearities, improve the accuracy of ROMs while controling hyperreduction errors, and construct the parameter sample set $S_N$ via greedy sampling.  

%By incorporating these high-order methods, we can achieve accurate and efficient approximations, even for highly nonlinear systems, enabling real-time computations or repeated queries in many engineering and scientific applications.

%This is possible because the high-order EIM can generate up to $N^3$ basis functions and interpolation points. 

%If there are other nonlinear terms due to  $\bm f(u_N(\bm \mu), \bm x, \bm \mu)$,  we need to compute the error estimates for these terms and define $\hat{\varepsilon}_{M,P} (\bm \mu)$ as the maximum value of all the error estimates. To calculate the error estimates, we must specify both $M$ and $P$. We choose $P = N$ and $M$ as a multiple of $N$. 

\begin{algorithm}
\begin{algorithmic}[1]
\REQUIRE{Training set $\Xi^{\rm train}$ and initial sample $S_N = \{\bm \mu_1, \ldots, \bm \mu_N\}$.}
\ENSURE{Updated sample $S_N$.}
\STATE{Execute the offline stage outlined in Algorithm 3.}
\STATE{Perform the online stage outlined in Algorithm 4 for each $\bm \mu \in \Xi^{\rm train}$.}
\STATE{Solve (\ref{eq2wq}) and (\ref{eq2wqq}) for each $\bm \mu \in \Xi^{\rm train}$.}
\STATE{Evaluate $\hat{\varepsilon}_{M,P} (\bm \mu)$ from (\ref{EIMbound3w}) for each $\bm \mu \in \Xi^{\rm train}$.}
\STATE{Find $\bm \mu_{N+1} = \max_{\bm \mu \in \Xi^{\rm train}} \hat{\varepsilon}_{M,P} (\bm \mu)$.}
\STATE{Set $S_N = S_N \cup \bm \mu_{N+1}$.}
\STATE{If $|\hat{\varepsilon}_{M,P} (\bm \mu_{N+1})| \le \epsilon$, then stop. Else go back to Step 1.
}
%\STATE{Otherwise, go back to Step 1.}
\end{algorithmic}
\caption{Greedy sampling for the GN-H method.}
\end{algorithm}

\subsection{A simple test case}

We present numerical results from a simple test case to compare the performance of three empirical interpolation methods: the original EIM, the first-order EIM (FOEIM), and the second-order EIM (SOEIM). The original EIM relies solely on function values, while FOEIM incorporates first-order partial derivatives to enhance approximation accuracy with minimal additional computational effort. SOEIM further extends this approach by utilizing both first- and second-order partial derivatives. The results show that higher-order derivatives in SOEIM provide significantly improved accuracy without requiring additional parameter points. 

The test case involves the following parametrized functions 
$$u(x, \mu) = \frac{x}{(\mu + 1) \left(1 + \sqrt{\frac{\mu+1}{\exp(62.5)}} \exp \left(\frac{125 x^2}{\mu+1} \right) \right) }, \quad g(u) = 1 - \frac{1}{(1 + u)^2}$$
in a physical domain $\Omega = [0,2]$ and parameter domain $\mathcal{D} = [0, 10]$. Figure \ref{ex1fig1} shows the plots of the nonlinear function $g$ for different values of $\mu$. For the greedy sampling method, we choose an initial sample $S_N = \{0, 5, 10\}$ and use a training sample $\Xi^{\rm train}$ of 100 parameter points distributed uniformly in the parameter space. The greedy algorithm iteratively selects new parameter points based on error estimates and stops when the maximum error estimate $\hat{\varepsilon}_{M,P} (\mu_{N+1})$ falls below the prescribed tolerance of $10^{-3}$. The error estimates of SOEIM with $M=6N$ and $P=N$ are used to guide the greedy sampling. In this case, convergence is achieved at $N=18$. Figure \ref{ex1fig1} demonstrates the decreasing trend of $\hat{\varepsilon}_{M,P} (\mu_{N+1})$ as $N$ increases, illustrating the effectiveness of the greedy sampling. The parameter points selected by the greedy sampling algorithm are listed in Table \ref{ex1tab1}. These points are notably clustered toward the left side of the parameter space, indicating that this region requires a denser sampling to accurately capture the parametric manifold. 

\begin{figure}[htbp]
	\centering
	\begin{subfigure}[b]{0.49\textwidth}
		\centering		\includegraphics[width=\textwidth]{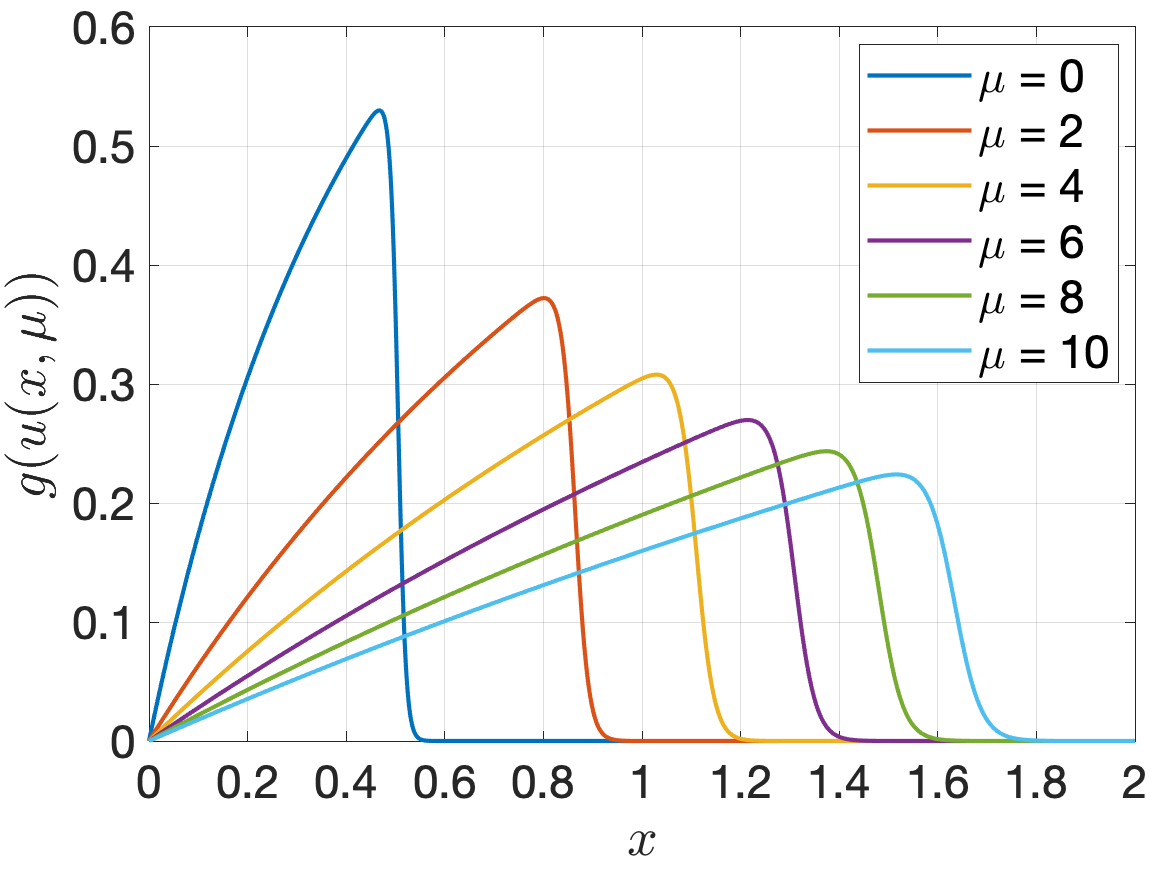}
		\caption{Instances of $g(u(x,\mu))$.}
	\end{subfigure}
	\hfill
	\begin{subfigure}[b]{0.49\textwidth}
		\centering		\includegraphics[width=\textwidth]{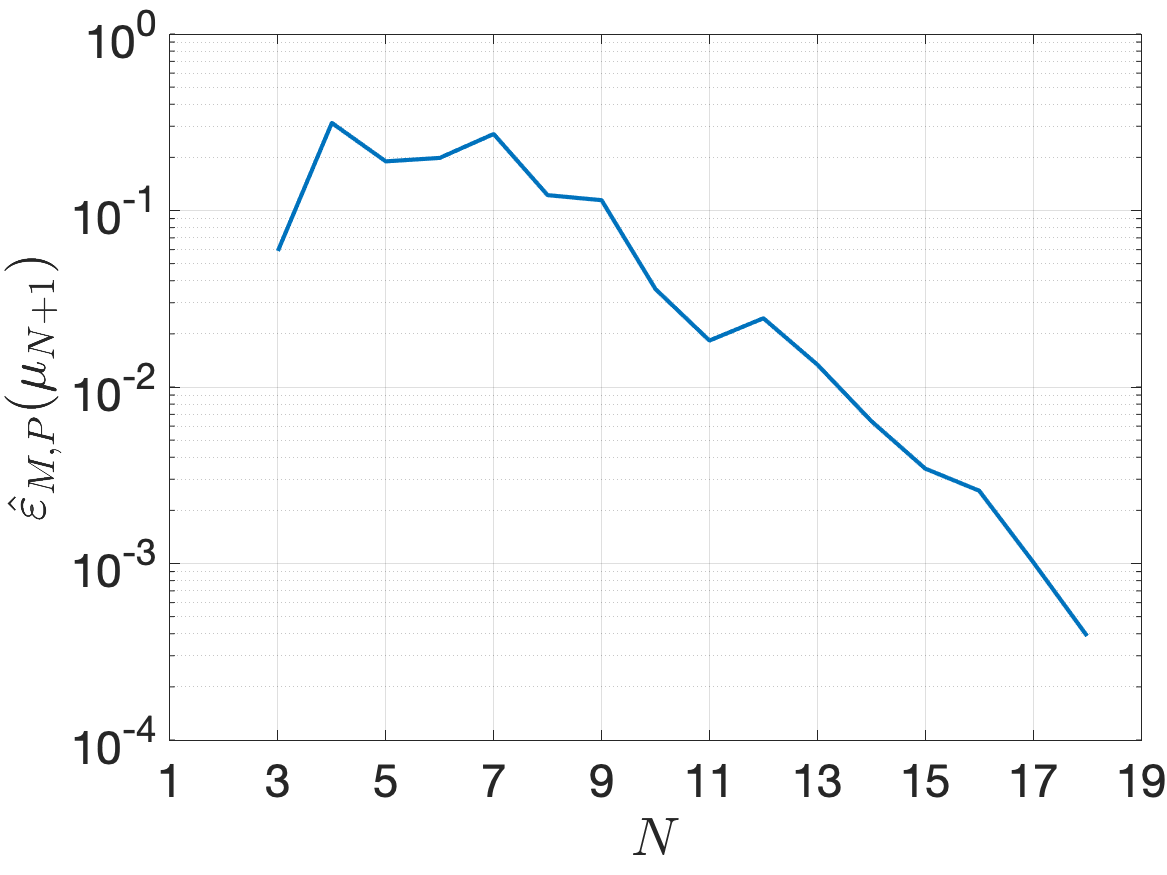}
		\caption{Maximum error estimate.}
	\end{subfigure}
	\caption{(a) Plots of $g(u(x,\mu))$ as a function of $x$ for  different values of $\mu$, and (b) the maximum error estimate $\hat{\varepsilon}_{M,P} (\mu_{N+1})$ in the greedy sampling.}
	\label{ex1fig1}
\end{figure}

 \begin{table}[h!]
\centering
\begin{tabular}{|c|c|c|c|c|c|c|c|}
\hline
$n$ & $\mu_n$ & $n$ & $\mu_n$ & $n$ & $\mu_n$ & $n$ & $\mu_n$ \\
\hline
4  & 0.4040  & 8  & 1.6162  & 12  & 0.2020  & 16  & 8.9899  \\ \hline
5  & 0.9091  & 9  & 7.9798  & 13  & 1.2121  & 17  & 5.6566  \\ \hline
6  & 6.5657  & 10 & 2.1212  & 14  & 0.6061  & 18  & 3.4343  \\ \hline
7  & 3.0303  & 11 & 3.9394  & 15  & 2.4242  & 19  & 4.3434  \\ \hline
\end{tabular}
\caption{Parameter points selected by the greedy sampling.}
\label{ex1tab1}
\end{table}

We use a test sample $\Xi^{\rm test}$ of $N^{\rm test} = 200$ parameter points distributed uniformly in the parameter space. The maximum error is defined as $ \varepsilon_M^{\max} = \max_{\mu \in \Xi^{\rm test}} \varepsilon_M(\mu)$ and the mean error as $\varepsilon_M^{\rm mean} = \frac{1}{N^{\rm test}}\sum_{\mu \in \Xi^{\rm test}} \varepsilon_M(\mu)$, where $\varepsilon_M(\mu)$ represents the interpolation error. Figure \ref{ex1fig2} plots $\varepsilon_M^{\rm max}$ and $\varepsilon_M^{\rm mean}$ as a function of $N$ for EIM, FOEIM, and SOEIM. These plots illustrate how each method performs in reducing both the worst-case and average errors as $N$ increases. Both FOEIM and SOEIM exhibit superior error reduction compared to EIM, with SOEIM showing errors several orders of magnitude lower than EIM. We define the mean error estimate as $\widehat{\varepsilon}_{M,P}^{\rm mean} = \frac{1}{N^{\rm test}}\sum_{\mu \in \Xi^{\rm test}} \widehat{\varepsilon}_{M,P}(\mu)$ and the mean effectivity as $\widehat{\eta}_{M,P}^{\rm mean} = \frac{1}{N^{\rm test}}\sum_{\mu \in \Xi^{\rm test}} \widehat{\varepsilon}_{M,P}(\mu)/\varepsilon_M(\mu)$. Tables \ref{ex1tab2} and \ref{ex1tab3} present the values of $\widehat{\varepsilon}_{M,P}^{\rm mean}$ and $\widehat{\eta}_{M,P}^{\rm mean}$ as a function $N$ and $M$ for FOEIM and SOEIM, respectively. 
The results indicate that the error estimates decrease rapidly as both $N$ and $M$ increase. Additionally, the error estimates are highly accurate, as reflected by the mean effectivities being consistently less than 5. This demonstrates that the estimation method provides sharp bounds on the actual error, ensuring efficiency, accuracy and reliability in the interpolation.

\begin{figure}[htbp]
	\centering
	\begin{subfigure}[b]{0.49\textwidth}
		\centering		\includegraphics[width=\textwidth]{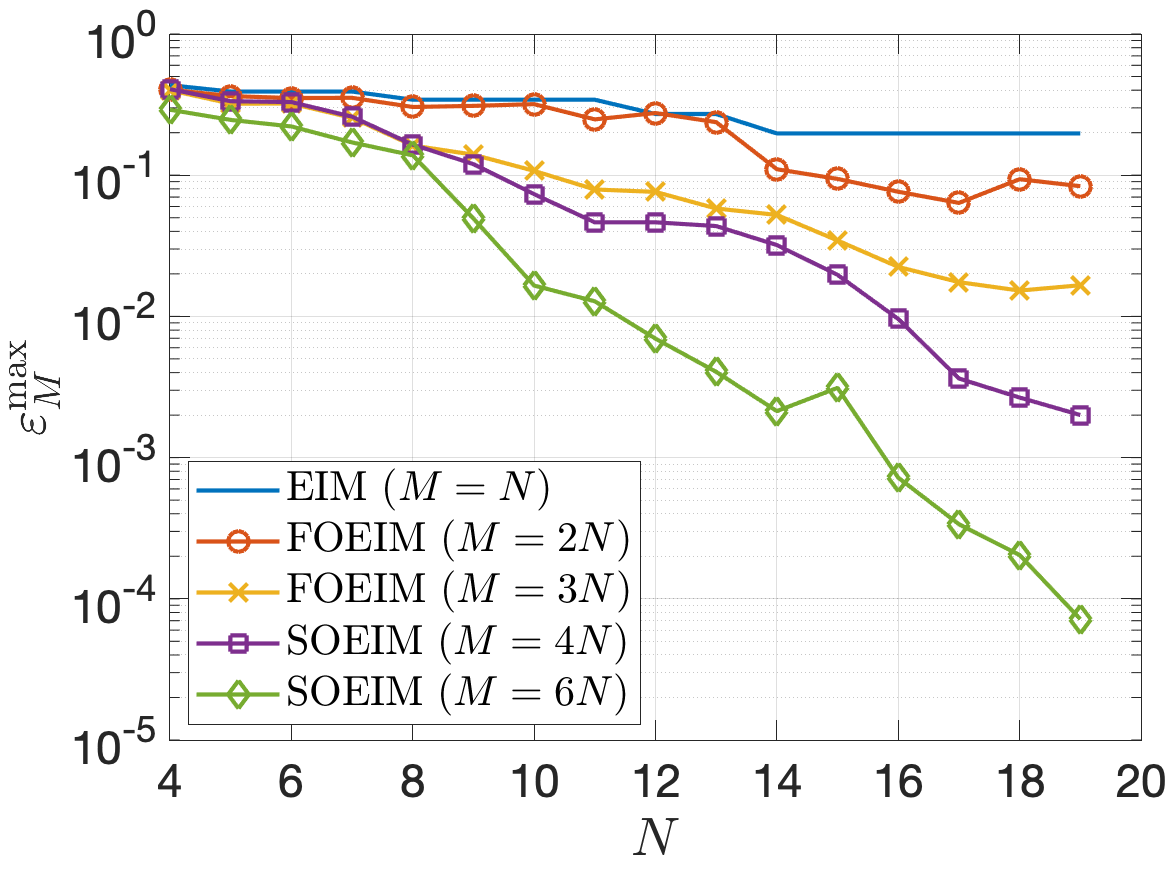}
		\caption{Maximum test error.}
	\end{subfigure}
	\hfill
	\begin{subfigure}[b]{0.49\textwidth}
		\centering		\includegraphics[width=\textwidth]{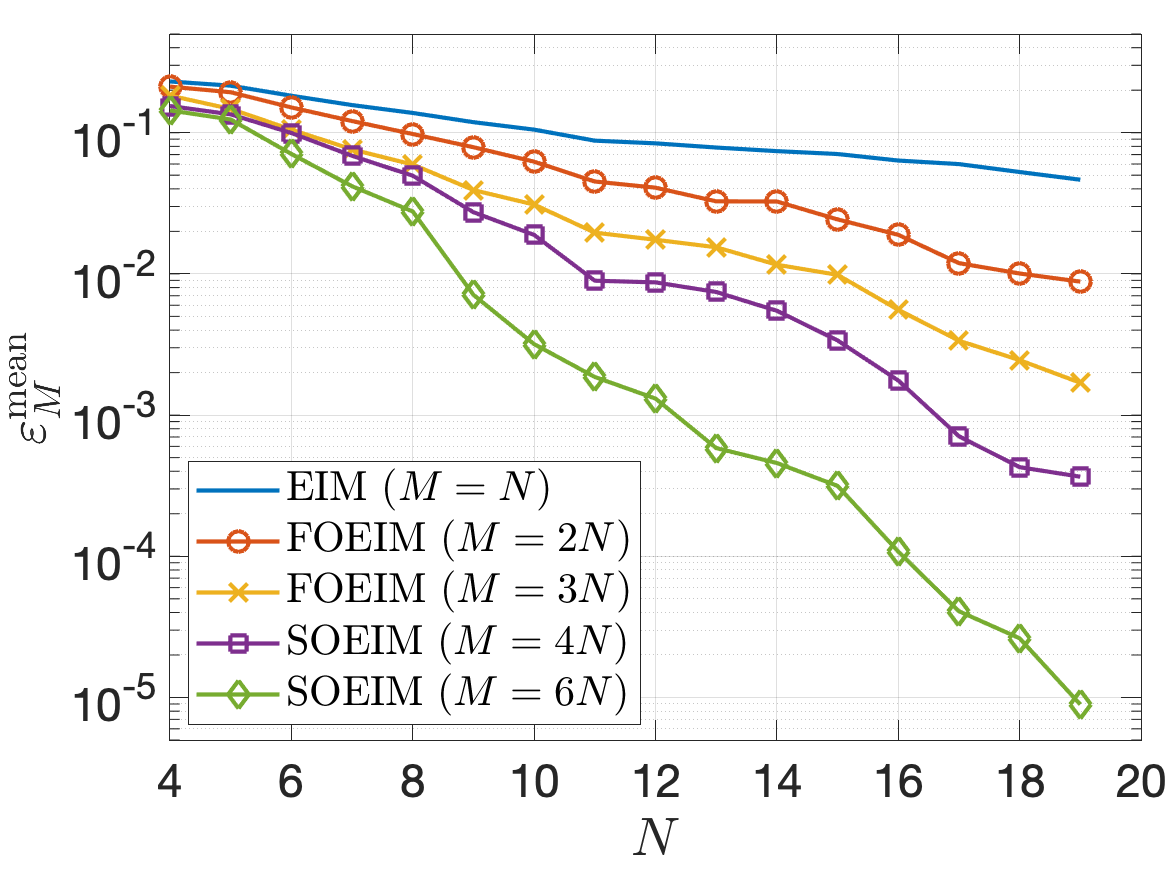}
		\caption{Mean test error.}
	\end{subfigure}
	\caption{Comparison of accuracy between EIM, FOEIM, and SOEIM: (a) the maximum test error and (b) the mean test error as a function of $N$.}
	\label{ex1fig2}
\end{figure}

\begin{table}[htbp]
\centering
\small
	\begin{tabular}{|c||cc|cc|cc|cc|}
		\cline{1-7}
    &		 
	 \multicolumn{2}{|c|}{$M=N$} & \multicolumn{2}{c|}{$M=2N$} & 
		 \multicolumn{2}{c|}{$M=3N$} \\   
   $N$ & $\widehat{\varepsilon}_{M,P}^{\rm mean}$ & $\widehat{\eta}_{M,P}^{\rm mean}$ & $\widehat{\varepsilon}_{M,P}^{\rm mean}$ & $\widehat{\eta}_{M,P}^{\rm mean}$ & $\widehat{\varepsilon}_{M,P}^{\rm mean}$ & $\widehat{\eta}_{M,P}^{\rm mean}$ \\
		\cline{1-7}
 4  &  1.64\mbox{e-}1  &   0.86  &  3.53\mbox{e-}1  &  1.61  &  3.16\mbox{e-}1  &  1.86  \\  
  7  &  1.72\mbox{e-}1  &  1.43  &  1.92\mbox{e-}1  &  1.47  &  7.93\mbox{e-}2  &  1.42  \\  
  10  &  1.68\mbox{e-}1  &  2.05  &  9.42\mbox{e-}2  &  1.48  &  7.45\mbox{e-}2  &  2.68  \\  
  13  &  1.50\mbox{e-}1  &  2.15  &  5.17\mbox{e-}2  &  1.84  &  3.44\mbox{e-}2  &  2.57  \\  
  16  &  1.38\mbox{e-}1  &  2.88  &  4.34\mbox{e-}2  &  2.57  &  1.33\mbox{e-}2  &  2.89  \\  
  19  &  1.07\mbox{e-}1  &  3.18  &  2.15\mbox{e-}2  &  2.83  &  7.23\mbox{e-}3  &  4.77   \\
		\hline
	\end{tabular}
	\caption{ $\widehat{\varepsilon}_{M,P}^{\rm mean}$ and $\widehat{\eta}_{M,P}^{\rm mean}$, where $P = N$, as a function of $N$ and $M$ for FOEIM.} 
	\label{ex1tab2}
\end{table}

\begin{table}[htbp]
\centering
\small
	\begin{tabular}{|c||cc|cc|cc|cc|}
		\cline{1-7}
    &		 
	 \multicolumn{2}{|c|}{$M=2N$} & \multicolumn{2}{c|}{$M=4N$} & 
		 \multicolumn{2}{c|}{$M=6N$} \\   
   $N$ & $\widehat{\varepsilon}_{M,P}^{\rm mean}$ & $\widehat{\eta}_{M,P}^{\rm mean}$ & $\widehat{\varepsilon}_{M,P}^{\rm mean}$ & $\widehat{\eta}_{M,P}^{\rm mean}$ & $\widehat{\varepsilon}_{M,P}^{\rm mean}$ & $\widehat{\eta}_{M,P}^{\rm mean}$ \\
		\cline{1-7}
  4  &  2.60\mbox{e-}1  &  1.19  &  1.13\mbox{e-}1  &   0.85  &  5.36\mbox{e-}2  &   0.45  \\  
  7  &  1.40\mbox{e-}1  &  1.26  &  8.13\mbox{e-}2  &  1.41  &  6.05\mbox{e-}2  &  1.61  \\  
  10  &  9.58\mbox{e-}2  &  1.62  &  2.44\mbox{e-}2  &  1.66  &  6.41\mbox{e-}3  &  2.12  \\  
  13  &  8.28\mbox{e-}2  &  2.28  &  2.01\mbox{e-}2  &  3.03  &  1.64\mbox{e-}3  &  3.00  \\  
  16  &  4.16\mbox{e-}2  &  2.11  &  4.28\mbox{e-}3  &  2.69  &  3.14\mbox{e-}4  &  2.95  \\  
  19  &  2.23\mbox{e-}2  &  2.67  &  8.43\mbox{e-}4  &  2.67  &  2.74\mbox{e-}5  &  3.44  \\  
		\hline
	\end{tabular}
	\caption{ $\widehat{\varepsilon}_{M,P}^{\rm mean}$ and $\widehat{\eta}_{M,P}^{\rm mean}$, where $P = N$, as a function of $N$ and $M$ for SOEIM.} 
	\label{ex1tab3}
\end{table}

Figure \ref{ex1fig3} compares the convergence of the mean error $\varepsilon_M^{\rm mean}$ as a function of $N$ for three sampling methods: Greedy Sampling, Extended Chebyshev, and Uniform Distribution. As $N$ increases, all methods demonstrate error reduction, with Greedy Sampling and Extended Chebyshev showing faster convergence and lower overall errors than Uniform Distribution. Greedy Sampling, in particular, exhibits superior performance at larger values of $N$, indicating its ability to better capture important features with fewer parameter points. Uniform Distribution converges more slowly, yielding the highest errors at larger values of $N$. A significant advantage of Greedy Sampling is its production of hierarchical samples, meaning new parameter points can be added incrementally to improve accuracy without reevaluating the entire set. This hierarchical nature makes Greedy Sampling highly flexible and efficient for ROM construction. In contrast, Extended Chebyshev, while competitive in terms of convergence rate, lacks the hierarchical refinement, limiting its adaptability.

\begin{figure}[htbp]
	\centering
 \includegraphics[width=0.7\textwidth]{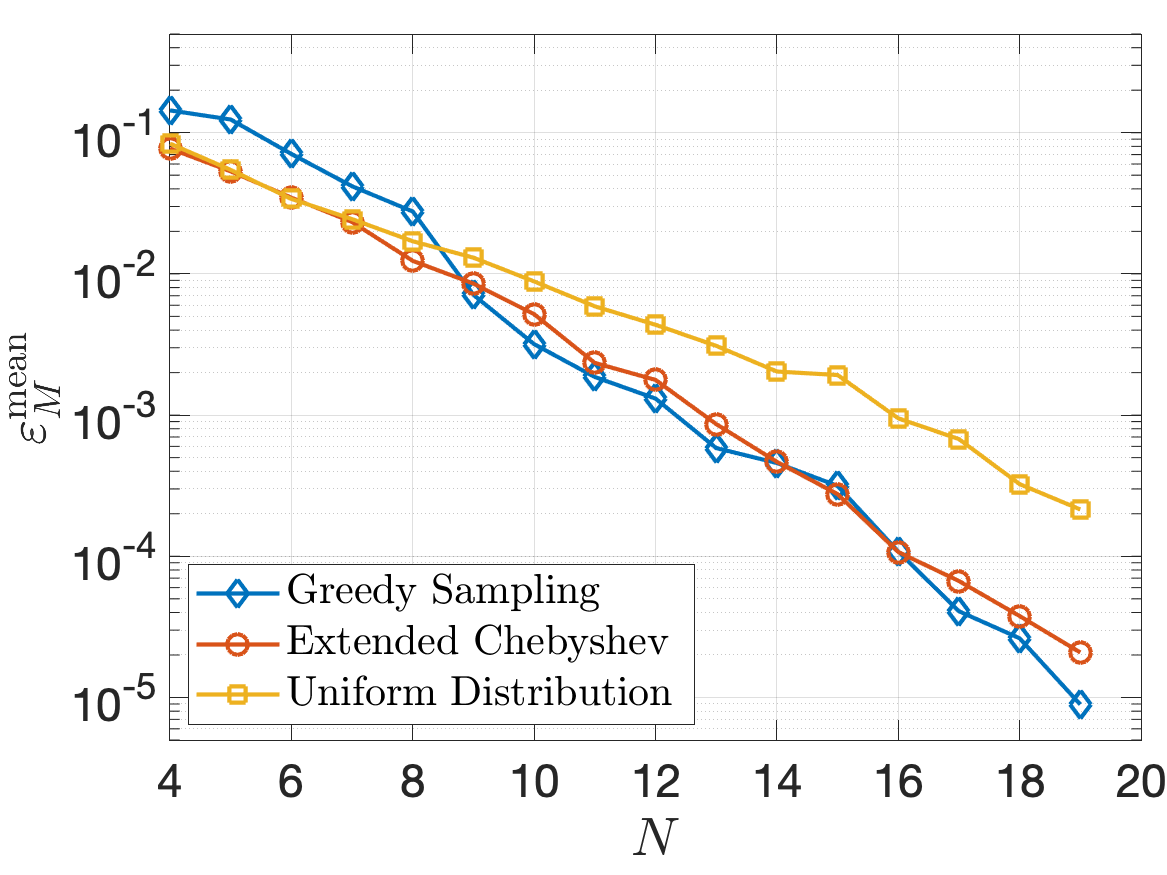}
 \caption{Convergence of the mean error $\varepsilon_M^{\rm mean}$ for SOEIM with $M=6N$ as a function of $N$ for Greedy Sampling, Extended Chebyshev, and Uniform Distribution.}
	\label{ex1fig3}
\end{figure}

\section{Numerical results}

In this section, we present numerical results from two parametrized nonlinear PDEs to assess and compare the performance of the H-GN and GN-H methods. Specifically, we investigate the impact of empirical interpolation methods on the convergence of these approaches relative to the standard GN method. For the H-GN method, we use three different interpolation methods -- EIM, FOEIM, and SOEIM -- to approximate the nonlinear terms, resulting in the EIM-GN, FOEIM-GN, and SOEIM-GN schemes, respectively.  In the GN-H method, SOEIM is applied to nonlinear terms in the residual vector, while FOEIM is used for the Jacobian matrix, producing the GN-SOEIM scheme. To balance computational efficiency and accuracy, we set different values of $M$ for each method. For EIM-GN, we use $M=N$, ensuring a basic level of approximation with the empirical interpolation method. For FOEIM-GN, 
$M=2N$ is used to leverage first-order partial derivatives, enhancing the approximation. For SOEIM-GN, we increase to $M=4N$ to incorporate second-order derivatives for more precise nonlinear term handling. In the GN-SOEIM method, we allocate $M=8N$ for the residual vector and $M=2N$ for the Jacobian matrix, allowing for higher accuracy in the residual while maintaining computational efficiency in the Jacobian approximation.

These methods are evaluated based on accuracy, computational cost, and convergence behavior. To assess the accuracy, we define the following  errors  
\begin{equation}
\epsilon^u_{N}(\bm \mu) = \|u(\bm \mu) - u_{N}(\bm \mu)\|_X, \qquad \epsilon^s_{N}(\bm \mu) = |s(\bm \mu) - s_{N}(\bm \mu)| ,
\end{equation}
and the average errors
\begin{equation}
\bar{\epsilon}_{N}^u =  \frac{1}{N^{\rm test}}\sum_{\bm \mu \in \Xi^{\rm test}} \epsilon_N^u (\bm \mu), \qquad  \bar{\epsilon}_{N}^s = \frac{1}{N^{\rm test}} \sum_{\bm \mu \in \Xi^{\rm test}}\epsilon_N^s (\bm \mu)  .  
\end{equation}  
where $\Xi^{\rm test}$ is a test sample of $N^{\rm test} = 900$ parameter points  distributed uniformly in the parameter domain. To measure the convergence of the H-GN and GN-H methods relative to the standard GN method, we define the effectivities 
\begin{equation}
\eta_{N}^u(\bm \mu) = \frac{\epsilon^u_{N}(\bm \mu)}{\epsilon^{u, \rm GN}_N(\bm \mu)}, \qquad \eta_{N}^s(\bm \mu) = \frac{\epsilon^s_{N}(\bm \mu)}{\epsilon^{s, \rm GN}_N(\bm \mu)},  
\end{equation}
where $\epsilon^{u, \rm GN}_N(\bm \mu)$ and $\epsilon^{s, \rm GN}_N(\bm \mu)$ are the errors for the GN method. The average effectivities $\bar{\eta}_{N}^u$ and $\bar{\eta}_{N}^s$ are similarly defined. 

%we use SOEIM and FOEIM to approximate the nonlinear terms in the resisdua, resulting in the EIM-GN, FOEIM-GN, and SOEIM-GN schemes, respectively.

%In this section, we present numerical results from two parametrized nonlinear PDEs to compare the performance of both the H-GN and GN-H methods.  We explore how the empirical interpolation methods impact the convergence of both the H-GN and GN-H methods relative to the standard GN method. When EIM, FOEIM, and SOEIM are used to approximate the nonlinear terms for the H-GN method, the resulting schemes are named H-GN-EIM, H-GN-FOEIM, and H-GN-SOEIM, respectively.    

%In the GN-H method, empirical interpolation is applied to the nonlinear terms in the residual vector and Jacobian matrix of the linearized system. The empirical interpolation applied to the residual vector is higher order than the one applied to the Jacobian matrix, allowing for improved accuracy in capturing nonlinearities. Specifically, if SOEIM is applied to the residual vector, then FOEIM is used for the Jacobian matrix.

% GN, H-GN (EIM), H-GN (FOEIM), H-GN (SOEIM), GN-H (SOEIM-FOEIM)

\subsection{A nonlinear elliptic problem}

We consider the following parametrized nonlinear elliptic PDE 
\begin{equation}
-\nabla^2u^{\rm e} + \mu_1 \exp(\sin(\mu_2 u^{\rm e})) = 100 \sin(2\pi x_{1})\cos(2 \pi x_{2}), \quad \mbox{in } \Omega, 
\end{equation} 
 with homogeneous Dirichlet condition on the boundary $\partial \Omega$, where $\Omega = (0,1)^2$ and $\bm \mu  \in {\cal D} \equiv [1, 2 \pi]^2$.  The output of interest is the average of the field variable over the
physical domain. The weak formulation is then stated as: given $\bm \mu \in
{\cal D}$, find $s(\bm \mu) = \int_\Omega u(\bm \mu)$, where $u(\bm \mu) \in X \subset H_0^1(\Omega) \equiv \{v \in
H^1(\Omega) \mbox{ } | \mbox{ } v|_{\partial \Omega} = 0\}$ is the solution
of
\begin{equation}
a(u(\bm \mu), v) + \int_\Omega g(u(\bm \mu), \bm \mu) \, v = f(v), \quad \forall  v
\in X \ , 
\label{eq:7-6}
\end{equation}
where
\begin{equation}
a(w, v) = \int_\Omega \nabla w \cdot \nabla v, \quad f(v) = 100
\int_\Omega \sin(2\pi x_{1}) \, \cos(2 \pi x_{2}) \, v,  
\label{eq:7-6a}
\end{equation}
and
\begin{equation}
g(u(\bm \mu), \bm \mu) = \mu_1  \exp(\sin(\mu_2 u(\bm \mu))).  
\label{eq:7-6b}
\end{equation}
The finite element (FE) approximation space is $X = \{v \in H_0^1(\Omega) : v|_K \in \mathcal{P}^3(T), \  \forall T \in \mathcal{T}_h \}$, where $\mathcal{P}^3(T)$ is a space of polynomials of degree $3$ on an element $T \in \mathcal{T}_h$ and $\mathcal{T}_h$ is a finite element grid of $32 \times 32$ quadrilaterals. The dimension of the FE space is $\mathcal{N} = 9409$.

Figure \ref{ex2fig1} illustrates the parameter points selected via greedy sampling. The plot reveals that the majority of the points are concentrated along the boundary of the parameter domain, with a notable clustering near the top boundary. This distribution suggests that the regions near the boundary, particularly the upper edge, exhibit greater variability or complexity in the solution manifold, requiring more refined sampling. The greedy sampling algorithm effectively targets these areas to ensure better approximation accuracy while minimizing the number of points needed to capture the solution manifold. Figure \ref{ex2fig1} displays the interpolation points selected by the SOEIM method for $N=15$ parameter sample points. The interpolation points are well-distributed across the physical domain, ensuring good coverage for the interpolation. However, it is notable that only one of the interpolation points is located directly on the boundary of the physical domain. This can be attributed to the fact that the solution vanishes to zero along the boundary, making boundary points less critical for capturing the essential variation of the solution within the domain.

\begin{figure}[htbp]
	\centering
	\begin{subfigure}[b]{0.49\textwidth}
		\centering		\includegraphics[width=\textwidth]{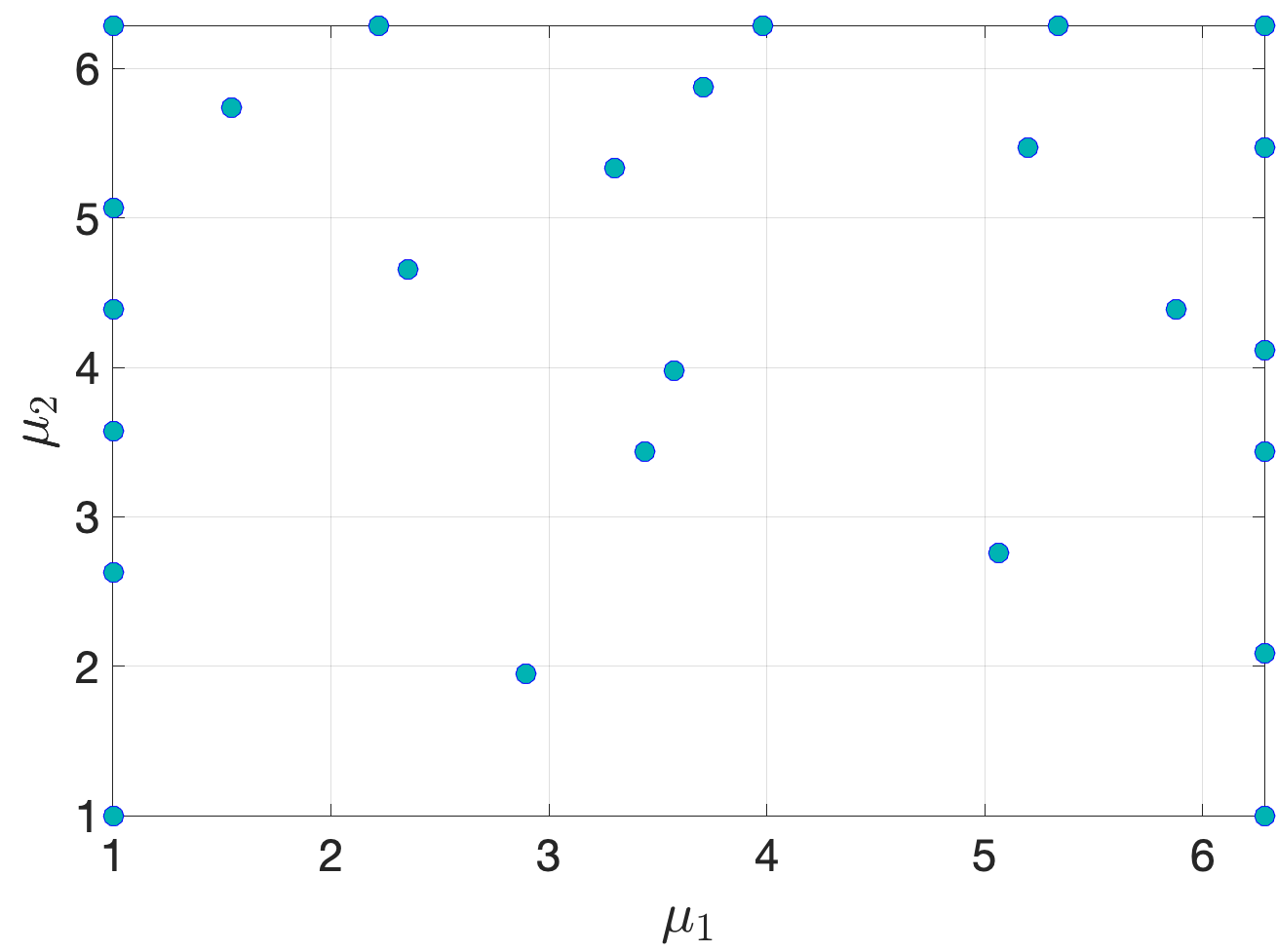}
		\caption{Parameter points in the sample $S_N$ for $N=25$.}
	\end{subfigure}
	\hfill
	\begin{subfigure}[b]{0.49\textwidth}
		\centering		\includegraphics[width=\textwidth]{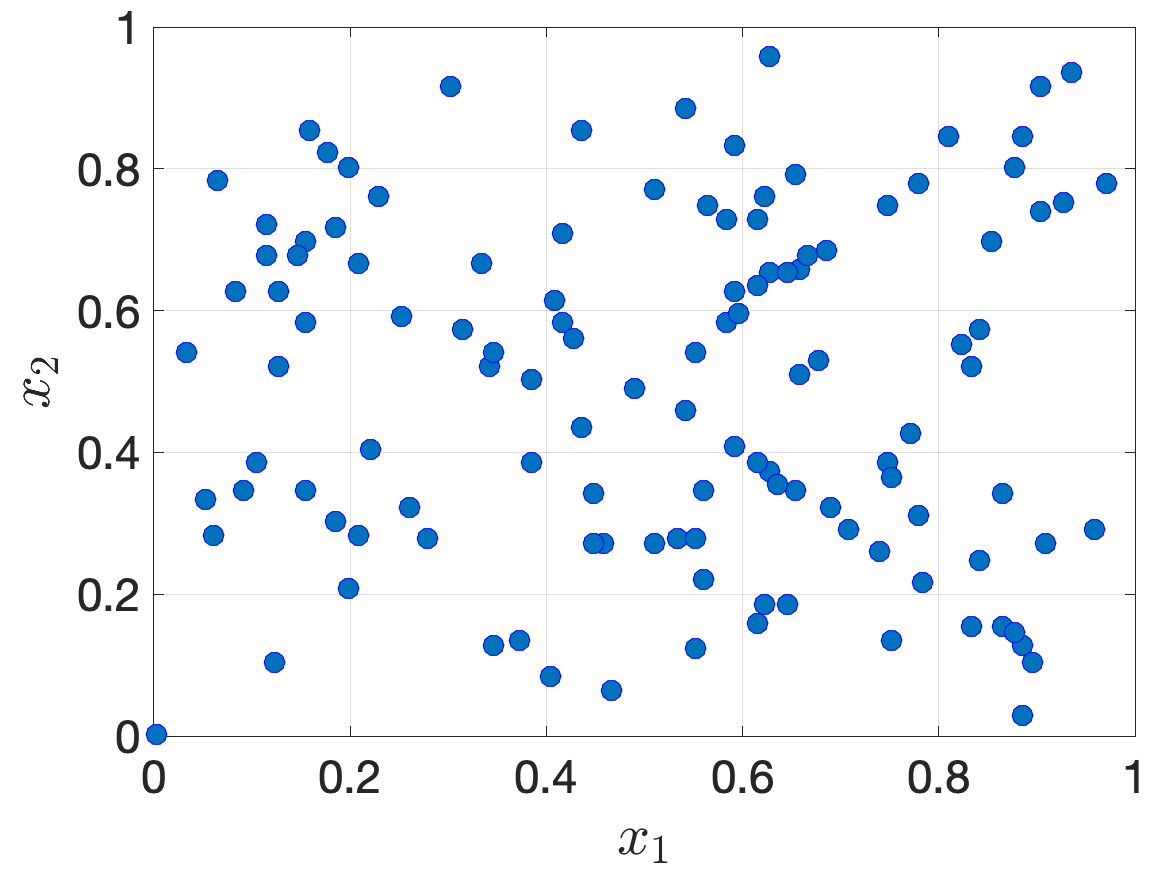}
		\caption{Interpolation points for $N=15$.}
	\end{subfigure}
	\caption{(a) distribution of the parameter sample points selected using the greedy sampling, and (b) distribution of the interpolation points for the SOEIM method for $N=15$.}
	\label{ex2fig1}
\end{figure}

Figure \ref{ex2fig2} presents the convergence of the mean solution error $\bar{\epsilon}_{N}^u$  and the mean output error $\bar{\epsilon}_{N}^s$ as functions of $N$ for five methods: EIM-GN, FOEIM-GN, SOEIM-GN, GN-SOEIM, and the standard GN method. As $N$ increases, all methods exhibit a clear reduction in error. In particular, FOEIM-GN, SOEIM-GN, and GN-SOEIM converge significantly faster than EIM-GN. The GN-SOEIM method closely mirrors the accuracy of the standard GN method, becoming almost indistinguishable for $N$ values greater than 10. This demonstrates that high-order interpolation significantly improves convergence, approaching the full GN performance while maintaining computational efficiency.

\begin{figure}[htbp]
	\centering
	\begin{subfigure}[b]{0.49\textwidth}
		\centering		\includegraphics[width=\textwidth]{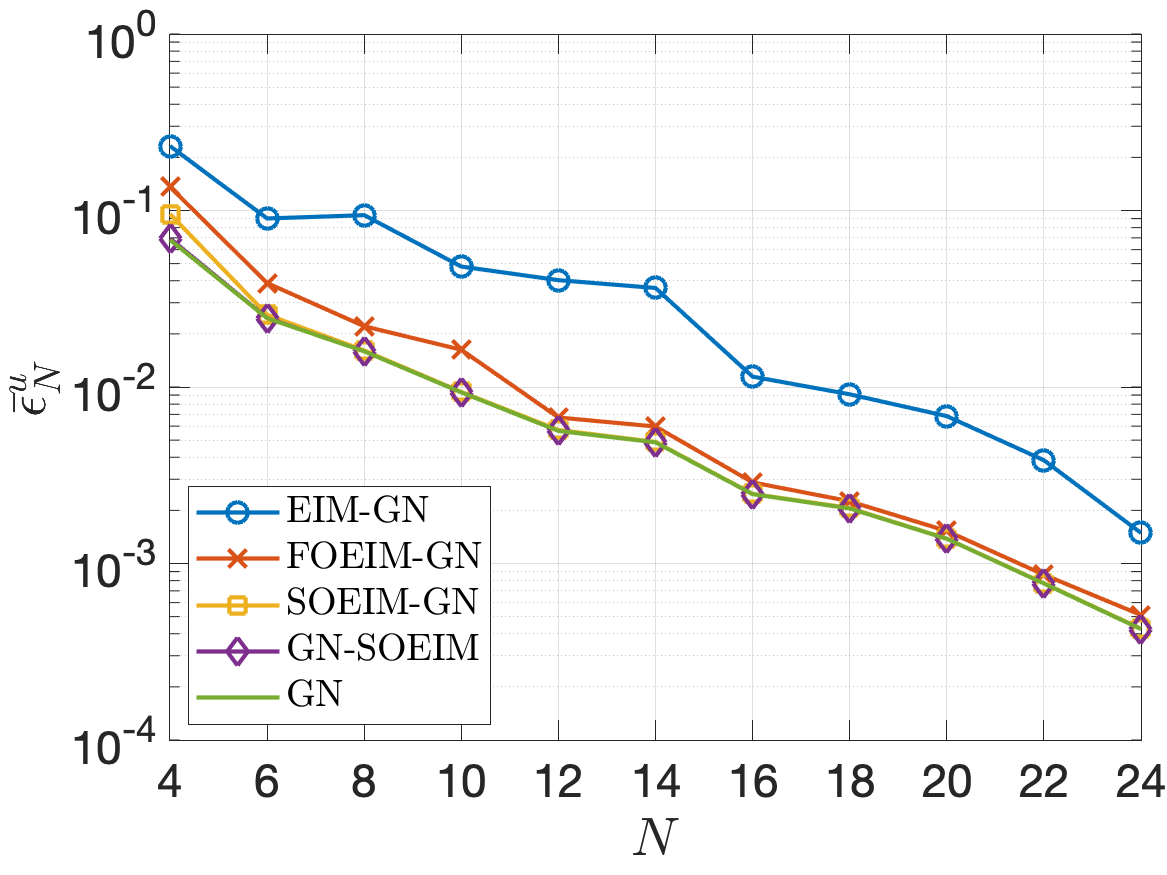}
		\caption{Average solution error $\bar{\epsilon}_{N}^u$.}
	\end{subfigure}
	\hfill
	\begin{subfigure}[b]{0.49\textwidth}
		\centering		\includegraphics[width=\textwidth]{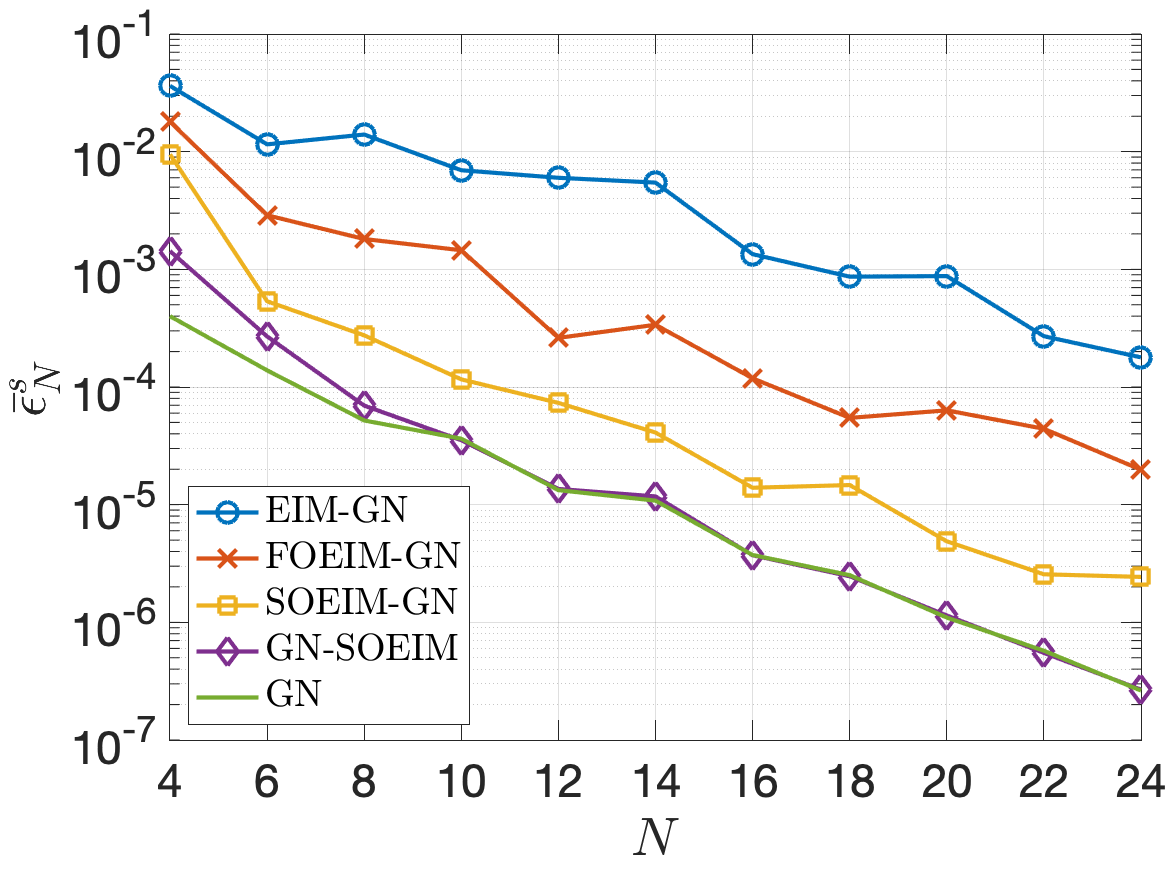}
		\caption{Average output error $\bar{\epsilon}_{N}^s$.}
	\end{subfigure}
	\caption{Comparison of accuracy between EIM-GN, FOEIM-GN, SOEIM-GN, GN-SOEIM, and GN methods.}
	\label{ex2fig2}
\end{figure}

Table \ref{ex2tab1} presents the average effectivities $\bar{\eta}_{N}^u$ (for the solution error) and 
$\bar{\eta}_{N}^s$ (for the output error), as a function of $N$ for EIM-GN, FOEIM-GN, SOEIM-GN, and GN-SOEIM. Across all values of $N$, GN-SOEIM consistently achieves the lowest effectivities. FOEIM-GN and SOEIM-GN perform significantly better than EIM-GN, especially as  $N$ increases. For GN-SOEIM, the  effectivities are very close to 1, indicating a very accurate and efficient approximation. In contrast, EIM-GN exhibits much larger effectivities, particularly in $\bar{\eta}_{N}^s$ reaching values over 1000 for higher $N$, demonstrating poorer performance in approximating the nonlinear terms. Notably, the output effectivties are considerably larger than the solution effectivities for EIM-GN, FOEIM-GN, and SOEIM-GN. This indicates that these methods struggle more with accurately approximating the output compared to the solution itself. Particularly in EIM-GN, the output effectivities grow substantially as  $N$ increases. In contrast, GN-SOEIM maintains consistently low output effectivities, demonstrating superior accuracy in capturing both solution and output errors. This suggests that higher-order interpolation methods, especially GN-SOEIM, are far more effective in improving output approximations.

\begin{table}[htbp]
\centering
\small
	\begin{tabular}{|c||cc|cc|cc|cc|cc|}
		\cline{1-9}
    &		 
	 \multicolumn{2}{|c|}{EIM-GN} & \multicolumn{2}{c|}{FOEIM-GN} & 
		 \multicolumn{2}{c|}{SOEIM-GN} &
		 \multicolumn{2}{c|}{GN-SOEIM} \\   
   $N$ & $\bar{\eta}_{N}^u$ & $\bar{\eta}_{N}^s$ & $\bar{\eta}_{N}^u$ & $\bar{\eta}_{N}^s$ & $\bar{\eta}_{N}^u$ & $\bar{\eta}_{N}^s$ & $\bar{\eta}_{N}^u$ & $\bar{\eta}_{N}^s$ \\
		\cline{1-9}
4  &  3.12  &  206.41  &  1.98  &  66.9  &  1.27  &  42.04  &  1.01  &  3.72  \\  
 6  &  4.02  &  325.04  &  2.03  &  63.95  &  1.06  &  6.49  &  1.00  &  1.51  \\  
 8  &  6.15  &  550.4  &  1.68  &  69.24  &  1.02  &  8.28  &  1.00  &  1.21  \\  
 10  &  5.44  &  490.66  &  2.03  &  75.66  &  1.02  &  6.94  &  1.00  &  1.22  \\  
 12  &  7.51  &  835.83  &  1.36  &  45.09  &  1.03  &  8.89  &  1.00  &  1.11  \\  
 14  &  7.04  &  777.49  &  1.54  &  59.95  &  1.03  &  6.25  &  1.00  &  1.26  \\  
 16  &  5.04  &  660.79  &  1.36  &  67.44  &  1.02  &  6.49  &  1.00  &  1.14  \\  
 18  &  4.90  &  607.98  &  1.34  &  55.63  &  1.03  &  8.05  &  1.00  &  1.03  \\  
 20  &  5.19  &  1153.0  &  1.38  &  101.82  &  1.01  &  8.40  &  1.00  &  1.13  \\  
 22  &  5.36  &  842.1  &  1.30  &  103.13  &  1.01  &  7.86  &  1.00  &  1.06  \\  
 24  &  4.02  &  1165.94  &  1.55  &  97.63  &  1.02  &  9.00  &  1.00  &  1.17  \\    
		\hline
	\end{tabular}
	\caption{Average effectivities as a function of $N$ for EIM-GN, FOEIM-GN, SOEIM-GN, and GN-SOEIM.} 
	\label{ex2tab1}
\end{table}

%We present in Table~\ref{tab2} the online computational times to calculate $s_{N}(\bm \mu)$ and $s_{N,M}(\bm \mu)$ as a function of $N$. The values are normalized with respect to the computational time  of the truth approximation output $s(\bm \mu)$. The computational saving is significant: for an relative accuracy of about 0.0001 ($N = 25$, $M = 200$) in the output, the reduction in online cost is more than a factor of 1000; this is mainly because the matrix assembly of the nonlinear terms for the truth approximation is computationally very expensive. The standard RB approximation has similar computational times as the truth FE approximation, and is between 100 and 1000 times slower than the RB approximation via empirical interpolation. \revise{We notice that using $M = N$ often requires more Newton iterations to converge than using $M > N$ especially when $N$ is relatively small. As a result, the online computational time with $M=N$ is slightly higher than that $M>N$ especially for $N = 9$ and $N=16$.} 

Table \ref{ex2tab2} shows the computational speedup for various model reduction methods -- GN, EIM-GN, FOEIM-GN, SOEIM-GN, and GN-SOEIM -- compared to the finite element method (FEM) for different values of $N$. Except for the standard GN method, the other model reduction methods provide substantial speedups relative to FEM. Across all values of $N$, the GN method achieves a modest speedup (around 2.2–2.7x) compared to FEM. In contrast, the EIM-GN method provides the highest speedup, ranging from over 5000x at $N=4$ to about 2400x at $N=24$, though its performance declines as 
$N$ increases. FOEIM-GN and SOEIM-GN demonstrate more consistent speedups (around 1700x–4200x), offering a good balance between computational efficiency and accuracy. GN-SOEIM maintains strong performance with speedup factors around 2000x–4600x. While both GN-SOEIM and FOEIM-GN exhibit similar speedup factors across a range of values for $N$, GN-SOEIM consistently outperforms FOEIM-GN in terms of accuracy. This superior performance stems from GN-SOEIM's use of second-order derivatives to approximate the residual and first-order derivatives for the Jacobian matrix, providing a more flexible hyperreduction strategy. The ability to balance these two interpolation techniques allows GN-SOEIM to effectively handle nonlinearities, offering a better trade-off between computational efficiency and accuracy.

\begin{table}[htbp]
\centering
\small
	\begin{tabular}{|c||c|c|c|c|c|}
		\cline{1-6}
  $N$  & \mbox{ } GN \mbox{ } & EIM-GN & FOEIM-GN  & SOEIM-GN & GN-SOEIM \\     
		\cline{1-6}
 4    &  2.52  &  5064.67  &  4258.1  &  3970.35  &  4601.2  \\  
 6    &  2.73  &  4052.52  &  4163.69  &  3555.97  &  4239.4  \\  
 8    &  2.76  &  3763.17  &  3938.08  &  3417.1  &  4480.89  \\  
 10   &  2.73  &  3589.21  &  3561.61  &  3537.51  &  4064.43  \\  
 12   &  2.66  &  3273.34  &  3467.02  &  3284.13  &  3813.88  \\  
 14   &  2.48  &  3049.29  &  3117.96  &  2838.7  &  3419.49  \\  
 16   &  2.40  &  3211.92  &  3011.17  &  2711.65  &  3390.79  \\  
 18   &  2.18  &  2566.5  &  2385.51  &  2076.51  &  2630.22  \\  
 20   &  2.26  &  2598.66  &  2412.79  &  2052.85  &  2391.55  \\  
 22   &  2.24  &  2427.95  &  2208.71  &  1795.99  &  2217.08  \\  
 24   &  2.27  &  2364.06  &  2122.91  &  1676.67  &  2070.6  \\    
		\hline
	\end{tabular}
	\caption{Computational speedup  relative to the finite element method (FEM) for different model reduction techniques (GN, EIM-GN, FOEIM-GN, SOEIM-GN, GN-SOEIM) as a function of $N$. The speedup is calculated as the ratio between the computational time of FEM and the online computational time of ROM.} 
	\label{ex2tab2}
\end{table}

\subsection{A nonlinear convection-diffusion-reaction problem}

We consider a  nonlinear convection-diffusion-reaction problem  
\begin{equation}
- \nabla^2u^{\rm e} +\nabla \cdot (\bm \mu (u^{\rm e})^2) +  g(u^{\rm e}) = 0, \quad \mbox{in } \Omega , 
\end{equation} 
with homogeneous Dirichlet condition $u^{\rm e}=0$ on  $\partial \Omega$. Here $\Omega = (0,1)^2$ is a unit square domain, while $\bm \mu = (\mu_1, \mu_2)$ is the parameter vector in a parameter domain $\mathcal{D} \equiv [0, 20] \times [0, 20]$. The reaction term is a nonlinear function of the state variable  $g(u^{\rm e}) = -2 \pi \exp(\sin(2 \pi  u^{\rm e}))$. 

Let $X \in H_0^1(\Omega)$ be a finite element (FE) approximation space of dimension $\mathcal{N}$ with $X = \{v \in H_0^1(\Omega) : v|_K \in \mathcal{P}^3(T), \  \forall T \in \mathcal{T}_h \}$, where $\mathcal{P}^3(T)$ is a space of polynomials of degree $3$ on an element $T \in \mathcal{T}_h$ and $\mathcal{T}_h$ is a finite element grid of $32 \times 32$ quadrilaterals. The dimension of the FE space $X$ is $\mathcal{N} = 9409$. The FE approximation $u(\bm \mu) \in X$ of the exact solution $u^{\rm e}(\bm \mu)$ is the solution of
\begin{equation}
\int_\Omega \nabla u \cdot \nabla v +  \int_\Omega \bm f(u, \bm \mu) \cdot \nabla v +  \int_\Omega g(u) \, v = 0, \quad \forall  v
\in X \ ,
\label{eq:7-6c}
\end{equation}
with $\bm f(u, \bm \mu) = -\bm \mu u^2$ being the nonlinear convection term. The output of interest is evaluated as $s(\bm \mu) = \ell^{O}(u(\bm \mu))$ with $\ell^{O}(v) \equiv \int_{\Omega} v$. Figure \ref{ex3fig1} shows four instances of $u(\bm x, \bm \mu)$ corresponding to the four corners of the parameter domain.

\begin{figure}[htbp]
	\centering
	\begin{subfigure}[b]{0.49\textwidth}
		\centering		\includegraphics[width=\textwidth]{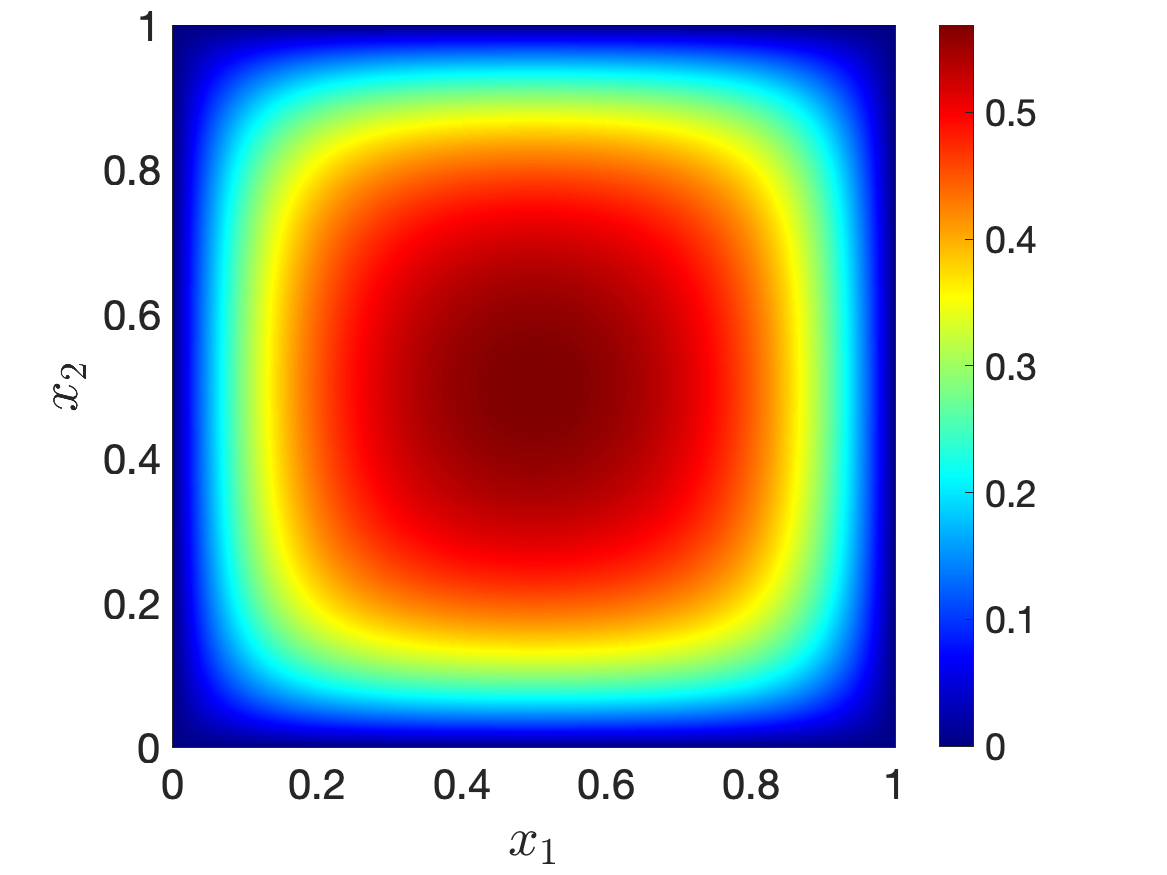}
		\caption{$\bm \mu = (0,0)$.}
	\end{subfigure}
	\hfill
	\begin{subfigure}[b]{0.49\textwidth}
		\centering		\includegraphics[width=\textwidth]{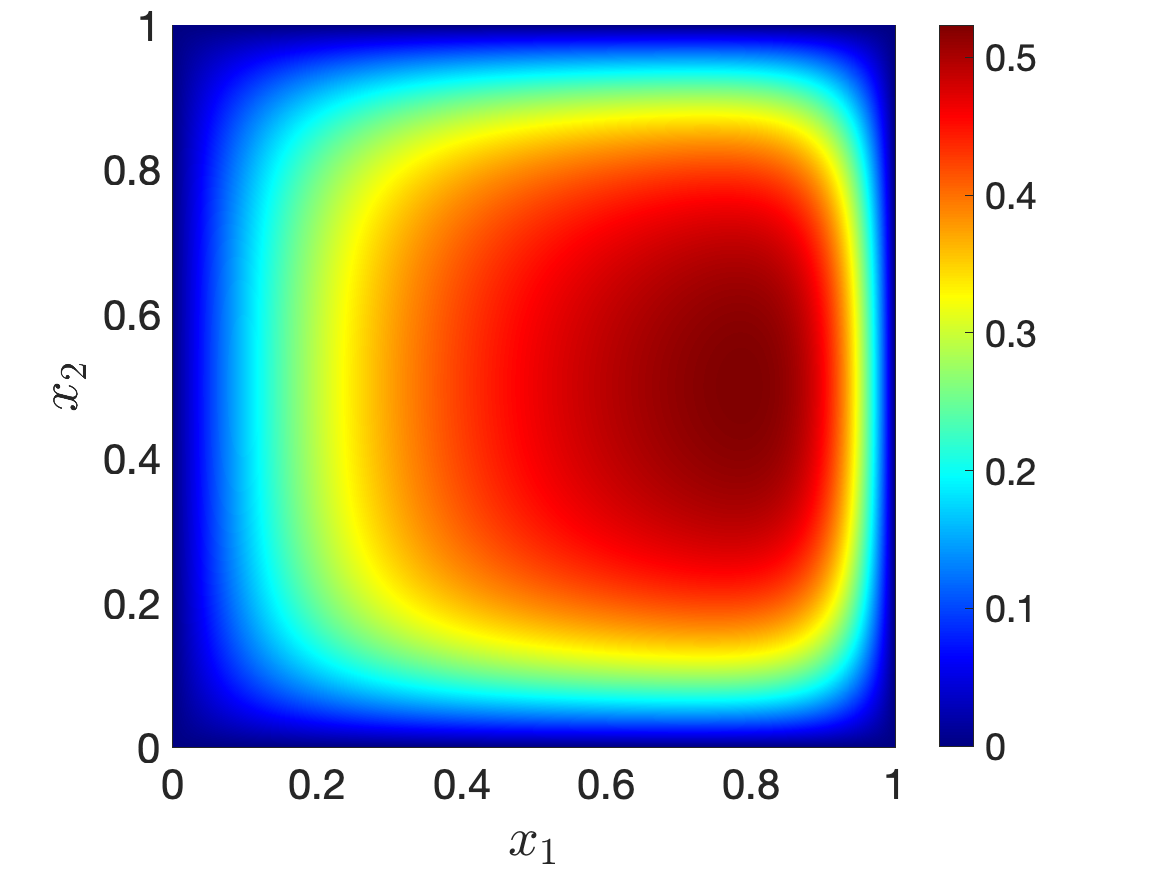}
		\caption{$\bm \mu = (20,0)$.}
	\end{subfigure} \\
        \begin{subfigure}[b]{0.49\textwidth}
		\centering		\includegraphics[width=\textwidth]{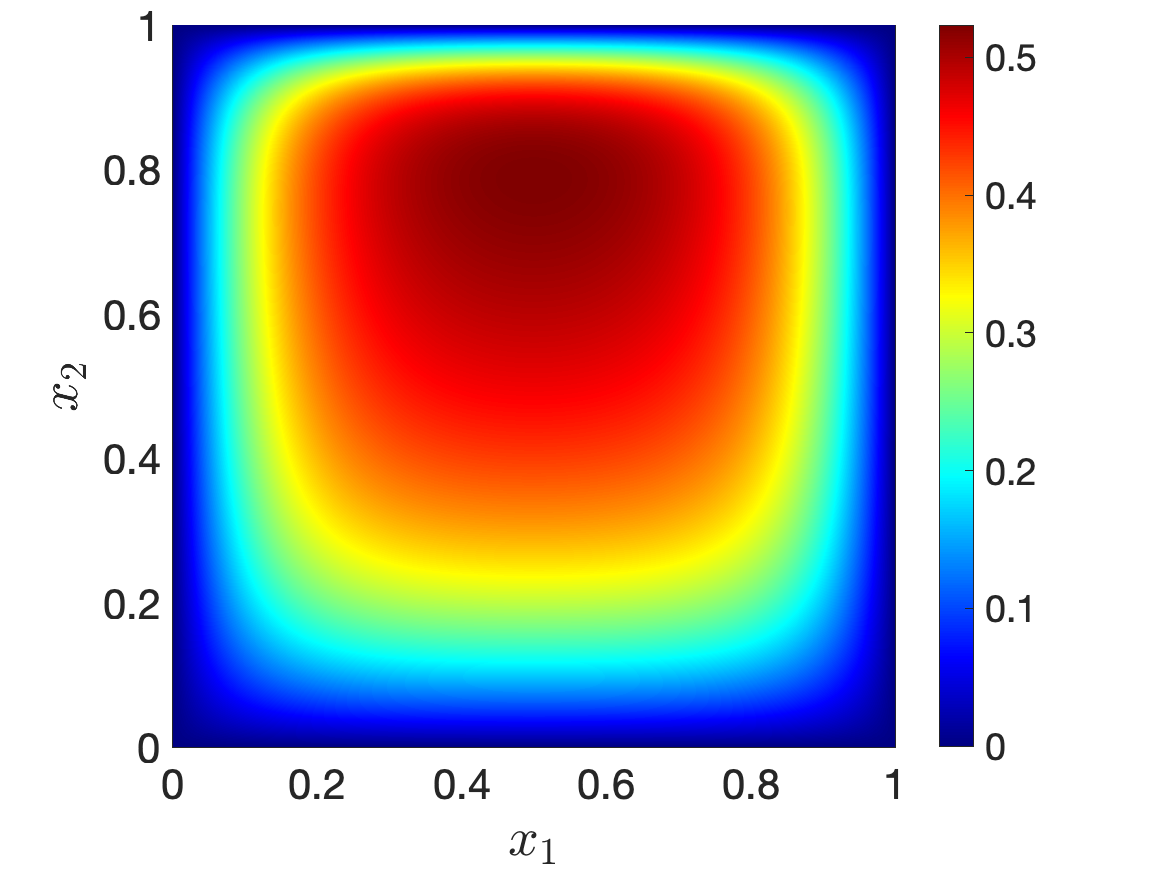}
		\caption{$\bm \mu = (0,20)$.}
	\end{subfigure}
	\hfill
	\begin{subfigure}[b]{0.49\textwidth}
		\centering		\includegraphics[width=\textwidth]{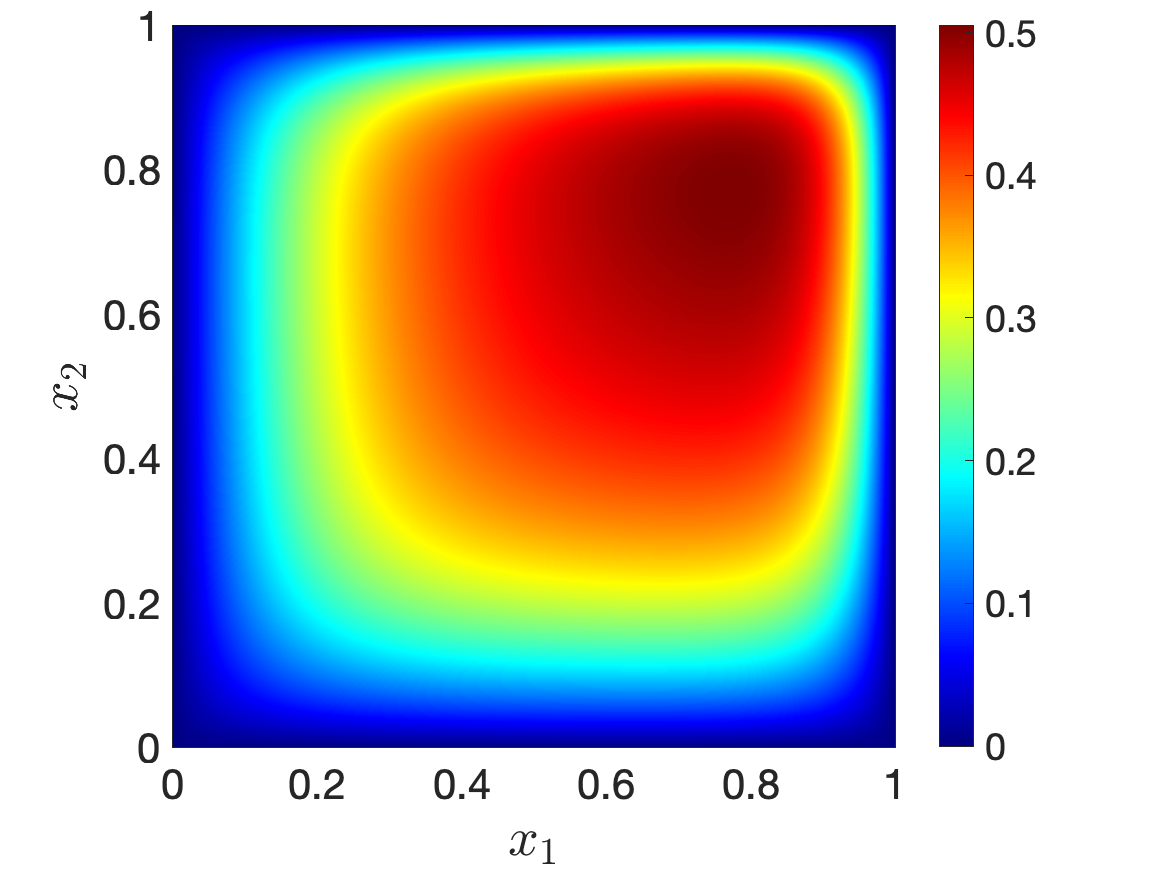}
		\caption{$\bm \mu = (20,20)$.}
	\end{subfigure}
	\caption{Four instances of the FOM solution $u(\bm x, \bm \mu)$.}
	\label{ex3fig1}
\end{figure}

Figure \ref{ex3fig2} illustrates the parameter points selected via greedy sampling. The plot reveals that the majority of the points are concentrated along the boundary of the parameter domain. This distribution suggests that the regions near the boundary exhibit greater variability or complexity in the solution manifold, requiring more refined sampling. Figure \ref{ex3fig2} displays the interpolation points selected by the SOEIM method for $N=25$ parameter sample points. The interpolation points are well-distributed across the physical domain. However, it is notable that only none of the interpolation points is located directly on the boundary of the physical domain since the solution vanishes to zero along the boundary, making boundary points less critical for capturing the essential variation of the solution within the domain. Figure \ref{ex3fig3} presents the convergence of the mean solution error $\bar{\epsilon}_{N}^u$  and the mean output error $\bar{\epsilon}_{N}^s$ as functions of $N$ for five methods: EIM-GN, FOEIM-GN, SOEIM-GN, GN-SOEIM, and the standard GN method. As $N$ increases, all methods tend to reduce the errors. The accuracy of the GN-SOEIM method is almost indistinguishable from that of the standard GN method. This demonstrates that high-order interpolation significantly improves convergence, approaching the full GN performance while maintaining computational efficiency.

\begin{figure}[htbp]
	\centering
	\begin{subfigure}[b]{0.49\textwidth}
		\centering		\includegraphics[width=\textwidth]{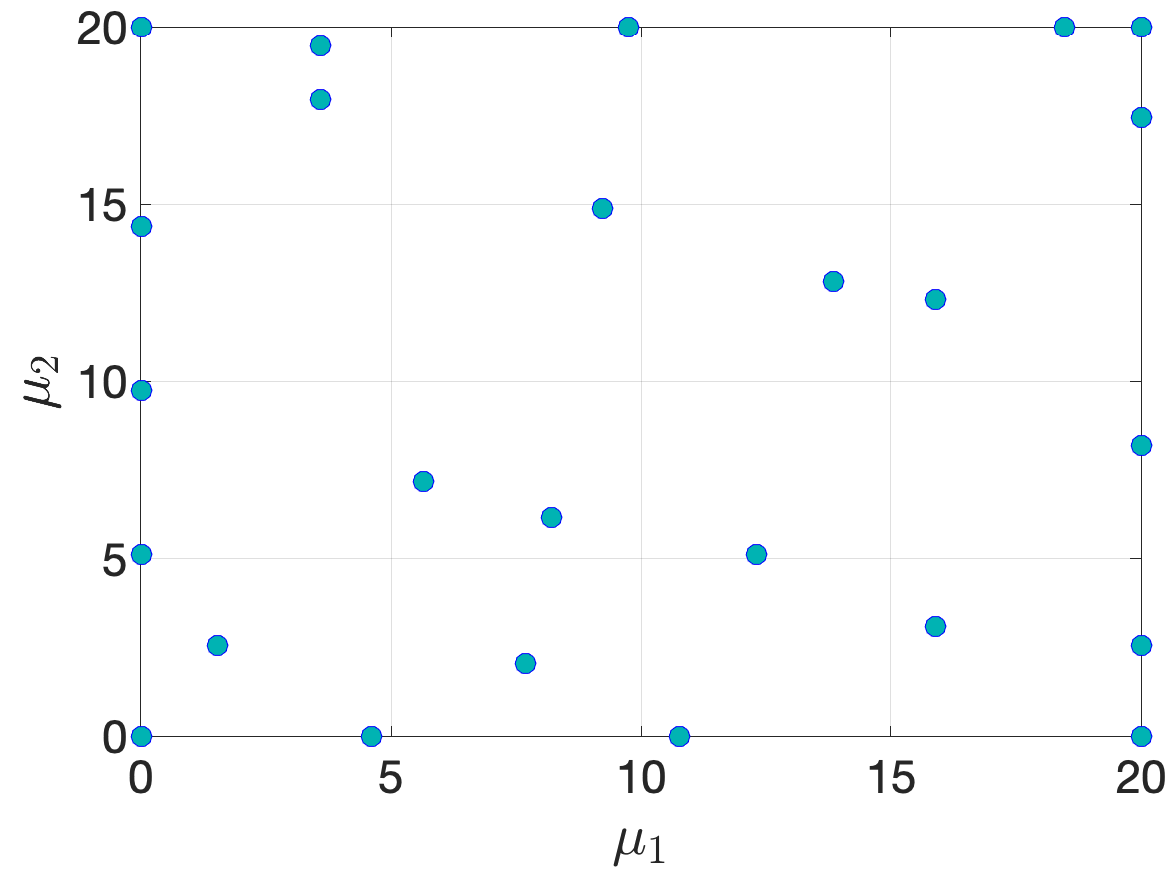}
		\caption{Parameter points in the sample $S_N$ for $N=25$.}
	\end{subfigure}
	\hfill
	\begin{subfigure}[b]{0.49\textwidth}
		\centering		\includegraphics[width=\textwidth]{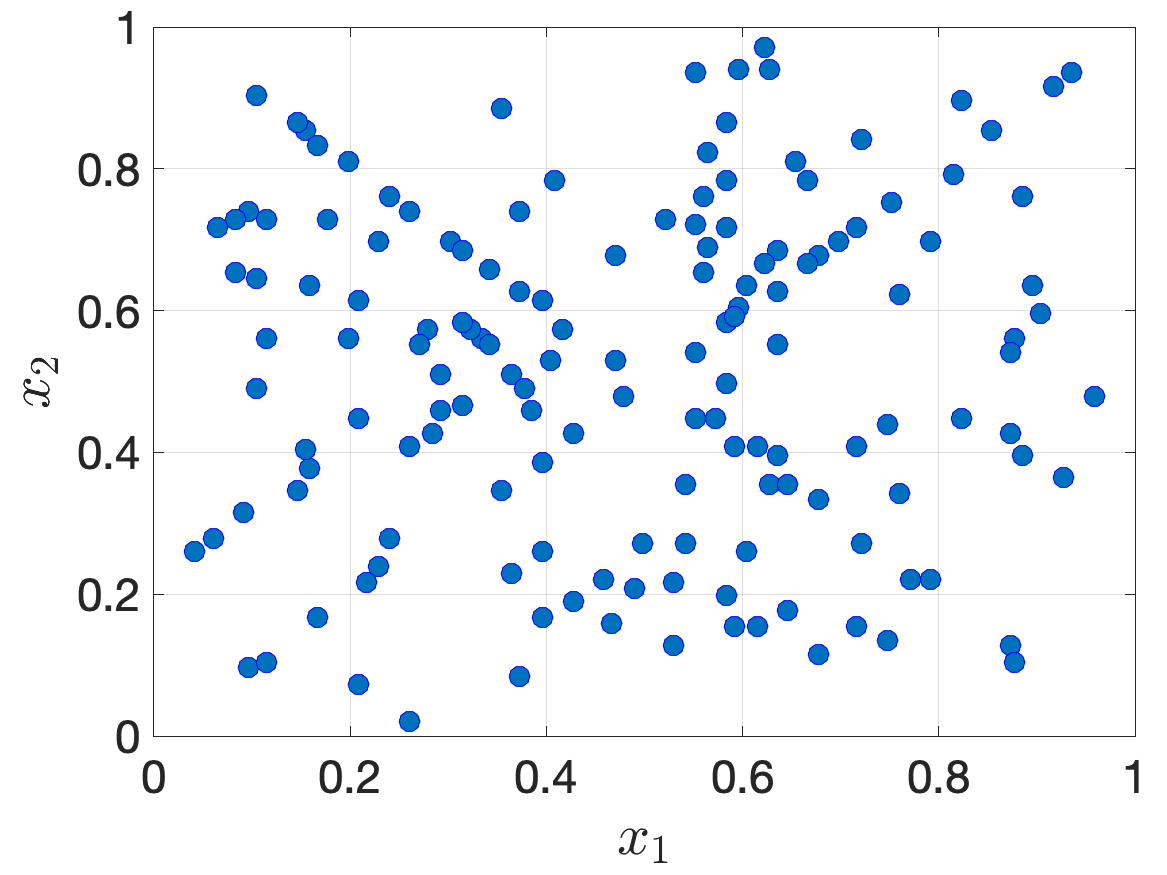}
		\caption{Interpolation points for $N=25$.}
	\end{subfigure}
	\caption{(a) distribution of the parameter sample points selected using the greedy sampling, and (b) distribution of the interpolation points for the SOEIM method for $N=25$.}
	\label{ex3fig2}
\end{figure}

\begin{figure}[htbp]
	\centering
	\begin{subfigure}[b]{0.49\textwidth}
		\centering		\includegraphics[width=\textwidth]{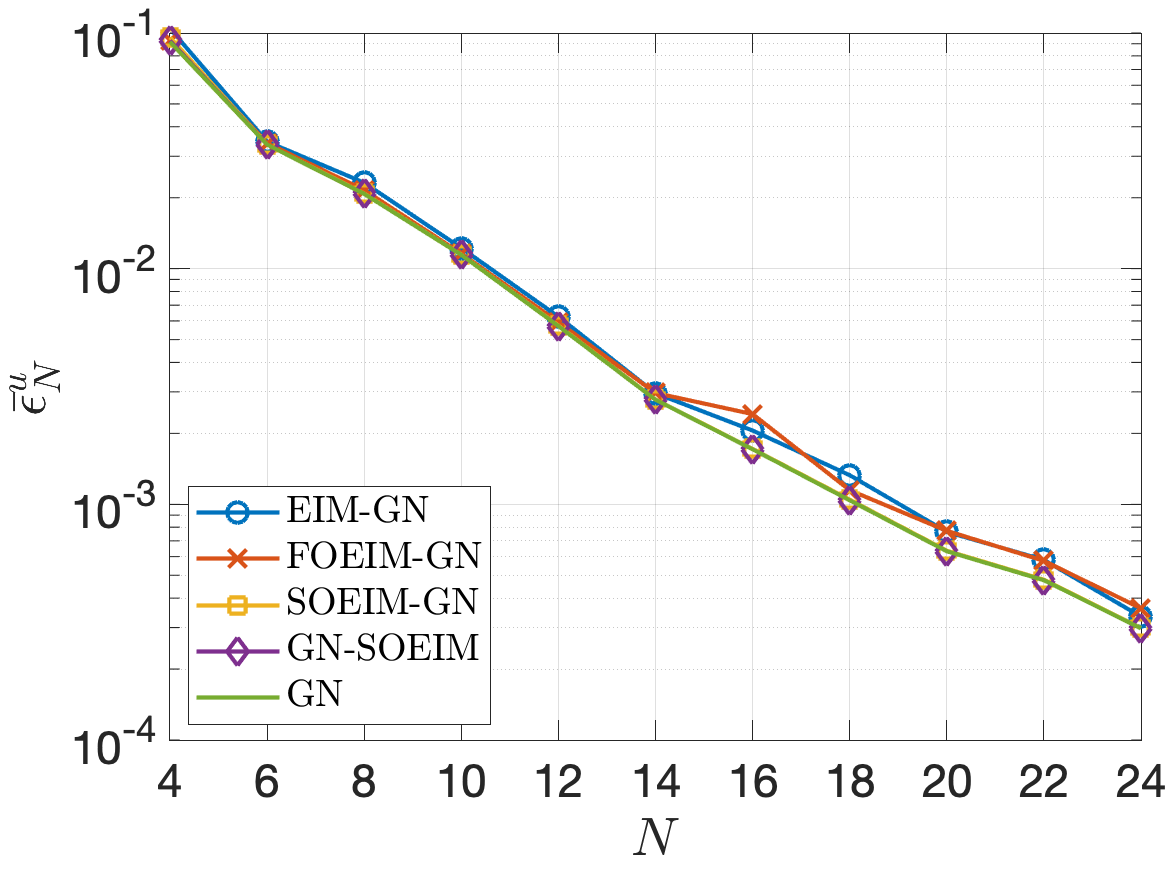}
		\caption{Average solution error $\bar{\epsilon}_{N}^u$.}
	\end{subfigure}
	\hfill
	\begin{subfigure}[b]{0.49\textwidth}
		\centering		\includegraphics[width=\textwidth]{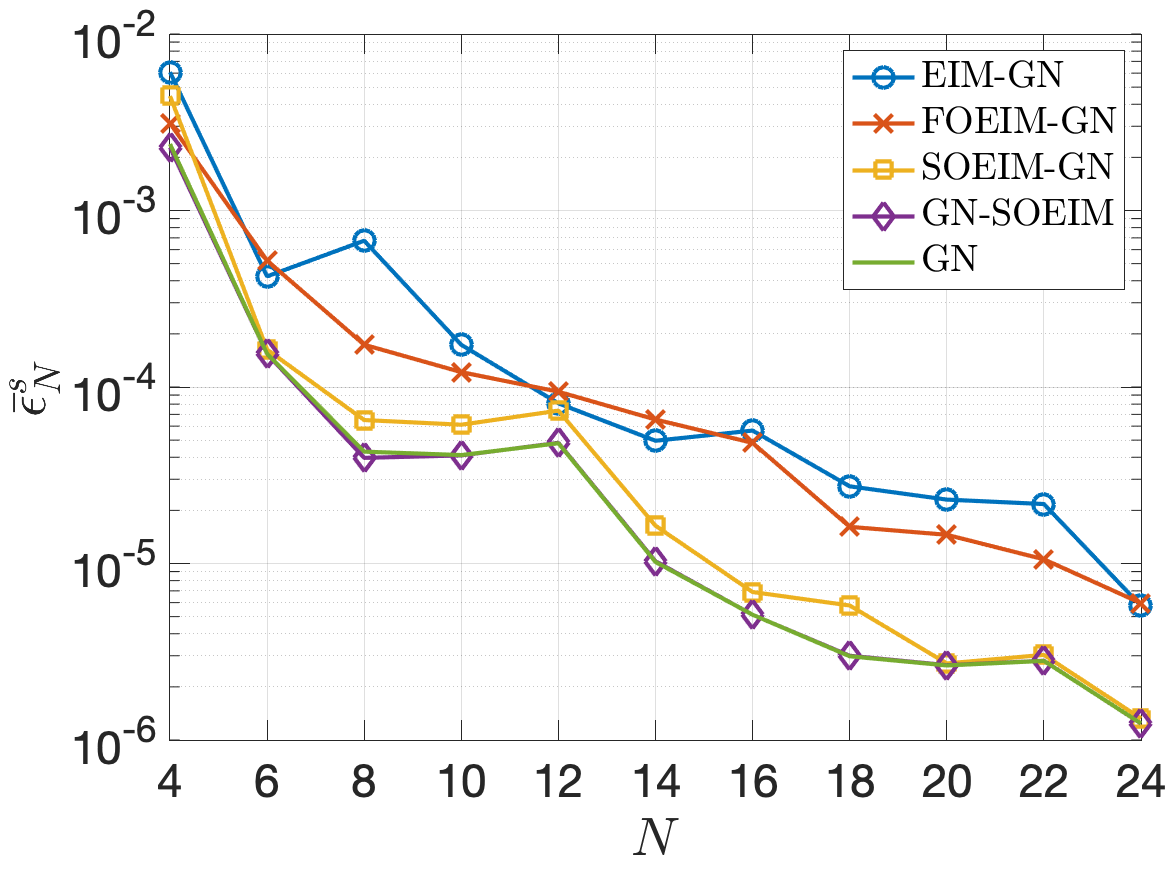}
		\caption{Average output error $\bar{\epsilon}_{N}^s$.}
	\end{subfigure}
	\caption{Comparison of accuracy between EIM-GN, FOEIM-GN, SOEIM-GN, GN-SOEIM, and GN methods.}
	\label{ex3fig3}
\end{figure}

Table \ref{ex3tab1} presents $\bar{\eta}_{N}^u$  and 
$\bar{\eta}_{N}^s$, as a function of $N$ for EIM-GN, FOEIM-GN, SOEIM-GN, and GN-SOEIM. SOEIM-GN performs better than both EIM-GN and FOEIM-GN. Across all values of $N$, GN-SOEIM consistently achieves the lowest effectivities very close to 1, indicating a very accurate and efficient approximation.  Notably, the output effectivties are considerably larger than the solution effectivities for EIM-GN and FOEIM-GN. In contrast, GN-SOEIM maintains consistently low output effectivities. Table \ref{ex3tab2} shows the computational speedup for various model reduction methods -- GN, EIM-GN, FOEIM-GN, SOEIM-GN, and GN-SOEIM -- compared to the finite element method (FEM) for different values of $N$. Except for the standard GN method, the other model reduction methods provide substantial speedups relative to FEM. Across all values of $N$, the GN method achieves a modest speedup (around 2.3–3.0x) compared to FEM. In contrast, the EIM-GN method provides the highest speedup over 5000x. While  GN-SOEIM exhibits similar speedup factors across a range of values for $N$ (around 5000x–7000x), it consistently outperforms other methods in terms of accuracy. This superior performance highlights the flexibility of GN-SOEIM in balancing computational efficiency and accuracy.

\begin{table}[htbp]
\centering
\small
	\begin{tabular}{|c||cc|cc|cc|cc|cc|}
		\cline{1-9}
    &		 
	 \multicolumn{2}{|c|}{EIM-GN} & \multicolumn{2}{c|}{FOEIM-GN} & 
		 \multicolumn{2}{c|}{SOEIM-GN} &
		 \multicolumn{2}{c|}{GN-SOEIM} \\   
   $N$ & $\bar{\eta}_{N}^u$ & $\bar{\eta}_{N}^s$ & $\bar{\eta}_{N}^u$ & $\bar{\eta}_{N}^s$ & $\bar{\eta}_{N}^u$ & $\bar{\eta}_{N}^s$ & $\bar{\eta}_{N}^u$ & $\bar{\eta}_{N}^s$ \\
		\cline{1-9}
4  &  1.13  &  2.60  &  1.02  &  1.33  &  1.04  &  2.05  &  1.01  &  0.97  \\  
 6  &  1.03  &  18.80  &  1.05  &  15.37  &  1.00  &  1.88  &  1.00  &  1.02  \\  
 8  &  1.32  &  28.15  &  1.10  &  10.78  &  1.01  &  3.38  &  1.00  &  0.99  \\  
 10  &  1.07  &  16.79  &  1.08  &  9.78  &  1.01  &  2.55  &  1.00  &  1.02  \\  
 12  &  1.10  &  10.65  &  1.12  &  6.67  &  1.02  &  3.21  &  1.00  &  1.02  \\  
 14  &  1.08  &  13.89  &  1.22  &  17.13  &  1.01  &  2.68  &  1.00  &  1.04  \\  
 16  &  1.22  &  43.72  &  1.84  &  24.98  &  1.01  &  2.37  &  1.00  &  1.01  \\  
 18  &  1.29  &  20.50  &  1.38  &  16.35  &  1.05  &  3.45  &  1.00  &  1.02  \\  
 20  &  1.23  &  31.65  &  1.68  &  19.02  &  1.05  &  2.18  &  1.00  &  1.01  \\  
 22  &  1.27  &  24.70  &  1.61  &  15.92  &  1.03  &  1.54  &  1.00  &  1.01  \\  
 24  &  1.14  &  21.92  &  1.55  &  16.16  &  1.02  &  1.61  &  1.00  &  1.02  \\  
		\hline
	\end{tabular}
	\caption{Average effectivities as a function of $N$ for EIM-GN, FOEIM-GN, SOEIM-GN, and GN-SOEIM.} 
	\label{ex3tab1}
\end{table}

\begin{table}[htbp]
\centering
\small
	\begin{tabular}{|c||c|c|c|c|c|}
		\cline{1-6}
  $N$  & \mbox{ } GN \mbox{ } & EIM-GN & FOEIM-GN  & SOEIM-GN & GN-SOEIM \\     
		\cline{1-6}
  4  &  2.94  &  7845.03  &  6865.92  &  5708.78  &  6982.76  \\  
 6  &  2.76  &  6933.33  &  6614.44  &   5268.42  &  6439.73  \\  
 8  &  2.73  &  6864.69  &  6637.31  &   4876.92  &  6266.16  \\  
 10  &  2.75  &  6607.40  &  6606.95  &  4633.53  &  6005.09  \\  
 12  &  2.77  &  6308.04  &  6215.00  &  4525.95  &  5828.64  \\  
 14  &  2.73  &  6183.81  &  6112.42  &  4392.77  &  5536.80  \\  
 16  &  2.71  &  6125.53  &  6073.42  &  4132.52  &  5490.81  \\  
 18  &  2.60  &  5971.05  &  5780.55  &   4061.84  &  5239.08  \\  
 20  &  2.44  &  5863.68  &  5626.32  &  3828.42  &  5173.65  \\  
 22  &  2.38  &  5739.28  &  5444.94  &  3772.29  &  5012.31  \\  
 24  &  2.43  &  5663.73  &  5354.86  &  3662.05  &  4977.26  \\  
		\hline
	\end{tabular}
	\caption{Computational speedup  relative to the finite element method (FEM) for different model reduction techniques (GN, EIM-GN, FOEIM-GN, SOEIM-GN, GN-SOEIM) as a function of $N$. The speedup is calculated as the ratio between the computational time of FEM and the online computational time of ROM.} 
	\label{ex3tab2}
\end{table}

\section{Conclusions}

This paper presents a class of high-order empirical interpolation methods to develop reduced-order models for parametrized nonlinear PDEs. The proposed methods significantly enhance the approximation power by leveraging higher-order partial derivatives to construct additional basis functions and interpolation points, providing greater accuracy and efficiency in reducing the dimensionality of nonlinear PDEs. Through numerical experiments, we demonstrated the superior convergence and accuracy of the high-order methods compared to traditional model reduction methods. These results suggest that the proposed methods can be widely applicable to a variety of problems, including those involving nonlinearities and nonaffine structures, while maintaining computational efficiency. Future work will focus on extending these methods to other nonlinear and time-dependent systems, as well as investigating their applicability to real-time control and optimization in engineering applications.

The development of model reduction methods is a very active area of research filled with open questions concerning the stability, accuracy, efficiency, and robustness of the reduced models. Model reduction methods can produce unstable or inaccurate solutions if they fail to preserve the geometric structures of the full model, such as conservation laws, symmetries, and invariants. To address this issue, structure-preserving model reduction techniques \cite{Afkham2017,Buchfink2019,Carlberg2015,Chaturantabut2016,Gong2017,Hesthaven2021,Lall2003,Sanderse2020,Schein2021} ensure that reduced models maintain key properties to improve both stability and accuracy. While hyper-reduction methods have successfully created efficient low-dimensional models, few approaches retain the structure of nonlinear operators such as the symplectic DEIM \cite{Peng2016} and energy-conserving ECSW scheme \cite{Farhat2015}. We would like to extend the high-order EIM to preserve geometric structures of the full model.

Model reduction on nonlinear manifolds was originated with the work \cite{Rutzmoser2017} on reduced-order modeling of nonlinear structural dynamics using quadratic manifolds. This concept was further extended through the use of deep convolutional autoencoders by Lee and Carlberg \cite{Lee2020}. Recently, quadratic manifolds are further developed to address the challenges posed by the Kolmogorov barrier in  model order reduction \cite{Geelen2023,Barnett2022}. Another approach is model reduction through lifting or variable transformations, as demonstrated by Kramer and Willcox \cite{Kramer2019}, where nonlinear systems are reformulated in a quadratic framework, making the reduced model more tractable. This method was expanded in the "Lift \& Learn" framework \cite{Qian2020} and operator inference techniques \cite{McQuarrie2023,Kramer2024}, leveraging physics-informed machine learning for large-scale nonlinear dynamical systems. We will combine our approach with nonlinear manifolds to develop more efficient ROMs for nonlinear PDEs.  

Model reduction methods applied to convection-dominated problems often struggle due to slowly decaying Kolmogorov $n$ widths caused by moving shocks and discontinuities. To address this issue, various techniques recast the problem in a more suitable coordinate frame using parameter-dependent maps, improving convergence. Approaches include POD-Galerkin methods with shifted reference frames \cite{Rowley2003}, optimal transport methods \cite{Iollo2014,Heyningen2023,Nguyen2023c}, and phase decomposition. Recent methods like shifted POD \cite{Reiss2018}, transport reversal \cite{Rim2018}, and transport snapshot \cite{Nair2019} use time-dependent shifts, while registration \cite{Taddei2020} and shock-fitting methods  \cite{Zahr2018,Zahr2020} minimize residuals to determine the map for better low-dimensional representations. We would like to extend this work to convection-dominated problems based on the optimal transport \cite{Heyningen2023,Nguyen2023c}.

With advances in data analytics and machine learning, data-driven ROMs have gained significant attention. Unlike projection-based ROMs, which rely on mathematical operators for Galerkin projection, data-driven ROMs infer expansion coefficients using regression models like radial basis functions \cite{Audouze2009,Ching2022,Heyningen2023,Walton2013,Xiao2015}, Gaussian processes \cite{Guo2018,Nguyen2016,Ortali2022}, or neural networks \cite{Hesthaven2018,Gao2021,Pichi2023}. These models treat the full-order model as a black box, making them easier to implement. However, they often lack error certification and require a large number of snapshots, which can be expensive for complex, high-dimensional systems. Future research will extend the proposed approach to data-driven model reduction.

A posteriori error estimation provides a critical advantage by ensuring confidence in the reduced model's accuracy and guiding the selection of the reduced model size based on accuracy needs. While rigorous error estimation techniques are well-developed for linear and weakly nonlinear PDEs \cite{Chen2009, deparis07, Eftang2012a, Huynh2007a, Huynh2010, PhuongHuynh2013, Karcher2018, Knezevic2011, Knezevic2010,  nguyen04:_handb_mater_model, Nguyen2007, Calcolo, ARCME, Rozza05_apnum,  Sen2006b, veroy04:_certif_navier_stokes, Veroy2002}, highly nonlinear PDEs present significant challenges. Dual-weighted residual (DWR) method enables error estimation in model reduction by using adjoint solutions to tailor error estimates for specific outputs \cite{meyer03:_effic_karhun_loeve,Nguyen2007,ARCME,Drohmann2015a,Carlberg2015b,Etter2020a,Yano2020,Du2021}. While DWR enhances reduced-order model (ROM) accuracy, it requires solving adjoint problems and is limited to first-order accuracy. To overcome these limitations, machine-learning error models \cite{Freno2019a,Blonigan2023} employ neural networks to predict errors. These models are more flexible but DWR remains more reliable for out-of-distribution errors. Future work will leverage these approaches to develop practical and reliable error estimates. 

\section*{Acknowledgements} \label{}
I would like to thank Professors Jaime Peraire, Anthony T. Patera, and Robert M. Freund at MIT, and Professor Yvon Maday at University of Paris VI for fruitful discussions. I gratefully acknowledge a Seed Grant from the MIT Portugal Program, the United States Department of Energy under contract DE-NA0003965 and the Air Force Office of Scientific Research under Grant No. FA9550-22-1-0356 for supporting this work.  

\bibliographystyle{elsarticle-num} 
\bibliography{library.bib}





% \end{thebibliography}
\end{document}